\documentclass[11pt, reqno]{amsart}

\usepackage{amssymb,latexsym,amsmath,amsfonts,amsthm}
\usepackage{mathrsfs}
\usepackage{enumitem}
\usepackage[usenames]{color}
\usepackage{hyperref}
\usepackage{comment}

\allowdisplaybreaks

\numberwithin{equation}{section}

\definecolor{DPurple}{rgb}{0.46,0.2,0.69}
\definecolor{Wine}{rgb}{0.7,0,0.2}

\theoremstyle{definition}
\newtheorem{definition}{Definition}[section]

\theoremstyle{remark}
\newtheorem{remark}[definition]{Remark}

\theoremstyle{plain}
\newtheorem{theorem}[definition]{Theorem}
\newtheorem{result}[definition]{Result}
\newtheorem{lemma}[definition]{Lemma}
\newtheorem{proposition}[definition]{Proposition}

\newtheorem{corollary}[definition]{Corollary}

\newcommand{\OM}{\Omega}
\newcommand{\D}{\mathbb{D}}

\newcommand{\smoo}{\mathcal{C}}

\newcommand{\emb}{\boldsymbol{{\sf j}}}

\newcommand{\bcdot}{\boldsymbol{\cdot}}
\newcommand{\lrarw}{\longrightarrow}
\newcommand{\btl}{\blacktriangleleft}
\newcommand{\bdy}{\partial}
\newcommand{\vbdy}{\partial_{{\rm visi}}}
\newcommand{\wvbdy}{\partial_{{\rm wvisi}}}
\newcommand{\lbdy}{\partial_{{\rm lg}}}

\newcommand{\N}{\mathbb{N}}

\newcommand{\Cn}{\mathbb{C}^n}

\newcommand{\C}{\mathbb{C}} 
\newcommand{\R}{\mathbb{R}}
\newcommand{\Omc}{{\overline{\Omega}}^{\infty}}

\newcommand{\wt}{\widetilde}

\begin{document}

\title[Visibility domains in complex manifolds]{Visibility domains relative to the Kobayashi distance \\
in complex manifolds}

\author{Rumpa Masanta}
\address{Department of Mathematics, Indian Institute of Science, Bangalore 560012, India}
\email{rumpamasanta@iisc.ac.in}

\begin{abstract}
In this paper, we extend the notion of visibility relative to the Kobayashi distance to domains in arbitrary
complex manifolds. Visibility here refers to a property resembling visibility in the sense of
Eberlein--O'Neill for Riemannian manifolds. Since it is difficult, in general, to determine whether
domains are Cauchy-complete with respect to the Kobayashi distance, we do not assume so here. We
provide many sufficient conditions for visibility. We establish a Wolff--Denjoy-type theorem in a very
general setting as an application. We also explore some connections between visibility and Gromov
hyperbolicity for Kobyashi hyperbolic domains in the above setting. 
\end{abstract}

\keywords{Complex manifolds, hyperbolic imbedding, Kobayashi metric, visibility}
\subjclass[2020]{Primary: 32F45, 32Q45, 53C23 ; Secondary: 32H50, 32Q05}

\maketitle

\vspace{-5mm}
\section{Introduction and statement of results}\label{S:intro}
In the past few years, Bharali--Zimmer introduced a property that may be viewed as a weak notion of
negative curvature for domains $\Omega\varsubsetneq\Cn$ viewed as metric spaces equipped with the
Kobayashi distance $K_\Omega$\,---\,first for bounded domains \cite{bharalizimmer:gdwnv17} and later for
unbounded domains \cite{bharalizimmer:gdwnv23}. This property of $\Omega\varsubsetneq\Cn$ has been
used to understand extensions and the behaviour of certain holomorphic mappings into $\Omega$, and also
to study properties of $K_\Omega$: see, for
instance, \cite{bharalizimmer:gdwnv17, bharalimaitra:awnovfoewt21, braccinikolovthomas:vkgcdrp22,
chandelmitrasarkar:nvwrkdca21, bharalizimmer:gdwnv23, sarkar:lkdavd23, nikilovoktenthomas:lgnvwrKdc24}.
In \cite{bharalizimmer:gdwnv17,
bharalizimmer:gdwnv23}, the property referenced above was called the visibility property with respect to
$K_\Omega$ because it is reminiscent of visibility in the sense of Eberlein--O'Neill
\cite{eberleinneill:vm73}. Very roughly, the visibility property is that all geodesics originating
sufficiently close to and terminating sufficiently close to two distinct points in $\bdy\Omega$ must bend
uniformly into $\Omega$.
\smallskip

Motivated by the usefulness of the visibility property, we extend this notion to Kobayashi
hyperbolic domains $\Omega\varsubsetneq X$ (i.e., domains on which the Kobayashi pseudodistance $K_\Omega$
is a distance) for arbitrary complex manifolds $X$, and explore some of its consequences.
\smallskip

Given a Kobayashi hyperbolic domain $\Omega$ in a complex manifold $X$, if 
$\dim_{\C}(X)\geq 2$, it is in general very difficult to determine whether the metric space
$(\Omega, K_{\Omega})$ is Cauchy-complete\,---\,even when $\Omega\varsubsetneq\Cn$\,---\,and, so,
whether $(\Omega, K_{\Omega})$ is a geodesic space. Thus, the formal definition of the visibility
property requires some care. In the process, we get two closely related notions of visibility\,---\,the
second notion being termed 
\emph{weak visibility with respect to $K_\Omega$}. We will formally define both these notions below.
Functionally, these two notions were introduced to study continuous extensions to $\overline D$ of
holomorphic mappings between domains $f:D\lrarw\Omega$ when $\Omega$ has the (weak) visibility property.
To quote Bharali--Zimmer \cite{bharalizimmer:gdwnv23}, for $\Omega$ as above:
\begin{itemize}
 \item  weak visibility with respect to $K_\Omega$ is the property that allows one to deduce continuous extension for
 isometries with respect to $K_D$ and $K_\Omega$.
 \item visibility with respect to $K_\Omega$ is the property that allows one to deduce
 continuous extension for continuous \textbf{quasi}-isometries with respect to $K_D$ and $K_\Omega$.
\end{itemize}
That being said, in this paper we shall not discuss results on continuous extension since
the proofs involving (target spaces that are) domains in $\Cn$ extend verbatim to the more general
setting of this paper. The results we shall present are (with perhaps one exception) such that their
proofs are \textbf{not} routine extensions of the arguments used for domains in $\Cn$.
\smallskip

\subsection{General theorems on visibility}\label{SS:general-visibility}
We first formally define the notion of visibility discussed above. Before we do so, recall the
discussion above on the difficulty of knowing whether $(\Omega, K_{\Omega})$ is
Cauchy-complete (for brevity: $\Omega$ is \emph{Kobayashi complete})
and, consequently, whether any pair of distinct points admits a geodesic joining them.
This difficulty is overcome by considering \emph{$(\lambda, \kappa)$-almost-geodesics} (where 
$\lambda\geq 1$ and $\kappa\geq 0$), which will serve as substitutes for geodesics. A
$(\lambda, \kappa)$-almost-geodesic is, essentially, a kind of ``nice'' quasi-geodesic: we refer the
reader to Section~\ref{SS:AG} for the precise definition. The relevance of these curves is as follows:
given any $\lambda\geq 1$ and $\kappa > 0$, for any two distinct points $x, y\in \Omega$, there exists
a $(\lambda, \kappa)$-almost-geodesic joining them. This is a consequence of
Proposition~\ref{P:almost-geodesic}. We can now give the following

\begin{definition}\label{D:visible}
Let $X$ be a complex manifold and $\Omega\varsubsetneq X$ be a Kobayashi hyperbolic domain.
\begin{itemize}[leftmargin=25pt]
 \item[$(a)$] Let $p,q\in\bdy\Omega$, $p\neq q$. We say that the pair $(p,q)$ satisfies the
 \emph{visibility condition with respect to $K_\Omega$} if there exist neighbourhoods $U_p$ of $p$ and
 $U_q$ of $q$ in $X$ such that
 $\overline{U_p}\cap\overline{U_q}=\emptyset$ and such that for each $\lambda\geq 1$ 
 and each $\kappa\geq 0$, there
 exists a compact set $K\subset \Omega$ such that the image of each $(\lambda,\kappa)$-almost-geodesic
 $\sigma:[0,T]\lrarw\Omega$ with $\sigma(0)\in U_p$ and $\sigma(T)\in U_q$ intersects $K$.
 \item[$(b)$] We say that $\bdy\Omega$ is \emph{visible} if every pair of points $p,q\in\bdy\Omega$,
 $p\neq q$, satisfies
 the visibility condition with respect to $K_\Omega$.
 \item[$(c)$] Let $p,q\in\bdy\Omega$, $p\neq q$. We say that the pair $(p,q)$ satisfies the \emph{weak
 visibility condition with respect to $K_\Omega$} if the condition in $(a)$ is satisfied \textbf{only}
 for $\lambda=1$. We say
 that $\bdy\Omega$ is \emph{weakly visible} if every pair of points $p,q\in\bdy\Omega$, $p\neq q$,
 satisfies the weak visibility condition with respect to $K_\Omega$.
\end{itemize}
\end{definition}

We first explore the question whether, with $\Omega$ as at the top of this paper,
$\bdy\Omega$ being locally visible is the same as $\bdy\Omega$ being visible. Firstly, we clarify what
``$\bdy\Omega$ being locally visible'' means.

\begin{definition}\label{D:locally-visible}
Let $X$ be a complex manifold and $\Omega\varsubsetneq X$ be a domain.
\begin{itemize}[leftmargin=25pt]
 \item[$(a)$]Let $p\in\bdy\Omega$. We say that 
 $\bdy\Omega$ is \emph{locally visible at $p$} if there exists a neighbourhood $U_p$ of $p$ in $X$
 such that $U_p\cap\Omega$ is connected and is a Kobayashi hyperbolic domain and, there exists a
 neighbourhood $V_p$ of $p$ in $X$, 
 $V_p\subseteq U_p$, such that every pair of distinct points
 $q_1,q_2\in V_p\cap\bdy\Omega$ satisfies the visibility condition with respect to
 $K_{U_p\cap\Omega}$.
 \item[$(b)$]
 We say that $\bdy\Omega$
 is \emph{locally visible} if $\bdy\Omega$ is locally visible at each boundary point $p\in\bdy\Omega$.
 \item[$(c)$] Let $p\in\bdy\Omega$. We say that 
 $\bdy\Omega$ is \emph{locally weakly visible at $p$} if the condition in $(a)$ is satisfied
 \textbf{only} for $\lambda=1$.
 We say that $\bdy\Omega$ is \emph{locally weakly visible} if $\bdy\Omega$ is locally weakly
 visible at
 each boundary point $p\in\bdy\Omega$.  
\end{itemize}
\end{definition}

The above question was first explored by Bracci \emph{et al.} in
\cite{braccigaussiernikolovthomas:lgvkd23} in the context of bounded Kobayashi complete domains in
$\Cn$. We extend their analysis to domains in arbitrary complex manifolds that are not necessarily
relatively compact and not necessarily Kobayashi complete. Also see \cite{chandelgoraimaitrasarkar:vposva24}
for a result in the spirit of
\cite{braccigaussiernikolovthomas:lgvkd23} for Kobayashi hyperbolic domains in $\Cn$ that are not assumed
to be Kobayashi complete.
Our result, however, is \textbf{not} a routine
extension of the analysis in \cite{braccigaussiernikolovthomas:lgvkd23}. To elaborate upon this point,
let us first state a definition.

\begin{definition}\label{D:hyp_imb}
Let $X$ be a complex manifold and let $Z$ be a connected complex submanifold in $X$. $Z$ is said to be
\emph{hyperbolically imbedded in $X$} if for each $x\in\overline{Z}$ and for each neighbourhood $V_x$ of
$x$ in $X$, there exists a neighbourhood $W_x$ of $x$ in $X$, $\overline{W_x}\subset V_x$, such that
\[
  K_Z(\overline{W_x}\cap Z, Z\setminus V_x)>0.
\] 
\end{definition}

\begin{remark}\label{rem:hyp-imbed}
Clearly, $Z$ is hyperbolically imbedded in $X$ if and only if, for $x,y\in\overline{Z}$ and $x\ne y$,
there exist neighbourhoods $V_x$ of
$x$ and $V_y$ of $y$ in $X$ such that $K_Z(V_x\cap Z, V_y\cap Z)>0$.  
\end{remark}

The notion of $\bdy\Omega$ being locally weakly visible at $p\in \bdy\Omega$ is the closer
of the two notions in Definition~\ref{D:locally-visible} to local visibility at $p\in \bdy\Omega$ as
considered in \cite{braccigaussiernikolovthomas:lgvkd23}. For either notion, the extension of
the arguments in \cite{braccigaussiernikolovthomas:lgvkd23} to our setting faces several obstacles. E.g.,
with $X$ as above, it is not true in general that one can fix a Hermitian metric $h$ on $X$ such that the
open balls given by the associated distance $d_h$ are Kobayashi hyperbolic. If we assume that
$\bdy\Omega$ is locally visible (resp., locally weakly visible), then, by augmenting this by the 
assumption
that $\Omega$ is hyperbolically imbedded, we are able to deal with the obstacles mentioned
and deduce that $\bdy\Omega$ is visible (resp., weakly
visble). But the additional assumption proves not to be unnatural because 
the converses of the latter statements hold true! This summarises the following two theorems.

\begin{theorem}\label{T:local-global-visibility}
Let $X$ be a complex manifold and let $\Omega\varsubsetneq X$ be a Kobayashi hyperbolic domain.
Then, the following conditions are equivalent:
\begin{itemize}[leftmargin=25pt]
 \item[$(a)$] $\bdy\Omega$ is locally visible and $\Omega$ is a hyperbolically imbedded domain,
 \item[$(b)$] $\bdy\Omega$ is visible.
\end{itemize}
\end{theorem}

\begin{theorem}\label{T:local-global-weak-visibility}
Let $X$ be a complex manifold and let $\Omega\varsubsetneq X$ be a Kobayashi hyperbolic domain.
Then, the following conditions are equivalent:
\begin{itemize}[leftmargin=25pt]
 \item[$(a)$] $\bdy\Omega$ is locally weakly visible and $\Omega$ is a hyperbolically imbedded domain,
 \item[$(b)$] $\bdy\Omega$ is weakly visible.
\end{itemize}
\end{theorem}

The next concept that we will examine in this paper is introduced by the following definition.

\begin{definition}\label{D:visible-point}
Let $X$ be a complex manifold and $\Omega\varsubsetneq X$ be a Kobayashi hyperbolic domain.
Let $p\in\bdy\Omega$. We say that $p$ is a \emph{visible point of $\bdy\Omega$} (resp., $p$ is
a \emph{weakly visible point of $\bdy\Omega$}) if there exists a neighbourhood $U_p$ of
$p$ in $X$ such that every pair of distinct points $q_1,q_2\in U_p\cap\bdy\Omega$ satisfies the
visibility condition (resp., the weak visibility condition) with respect to $K_\Omega$.
Let $\vbdy\Omega$ denote the set of all visible points of $\bdy\Omega$ and
$\wvbdy\Omega$ denote the set of all weakly visible points of $\bdy\Omega$. 
\end{definition}

The pair of sets defined above are important, but technical, objects that act along with more
concrete notions to give us various sufficient conditions such that, given a domain
$\Omega\varsubsetneq X$, $X$ and $\Omega$ as in the above definition, $\bdy\Omega$ is visible (resp.,
weakly visible). The principal results that rely on Definition~\ref{D:visible-point} are stated later
in this section or in Section~\ref{S:compacti-visi} and which will be proved in
Sections~\ref{S:compacti-visi} and~\ref{S:sufficint-condition}. An easy consequence of the results in
Section~\ref{S:compacti-visi} is as follows.

\begin{proposition}\label{P:visible-points-empty}
Let $X$ be a complex manifold and let $\Omega\varsubsetneq X$ be a Kobayashi hyperbolic domain.
Each of the sets $\bdy\Omega\setminus \vbdy\Omega$ and $\bdy\Omega\setminus \wvbdy\Omega$ is either not
totally disconnected or is empty. 
\end{proposition}

\subsection{Sufficient conditions for (weak) visibility}\label{SS:sufficient}
Next, we show that domains $\Omega\varsubsetneq X$ such that $\bdy\Omega$ is (weakly) visible are abundant.

\begin{definition}\label{D:visible-around-p}
Let $X$ be a complex manifold and $\Omega\varsubsetneq X$ be a Kobayashi hyperbolic domain. Let
$p\in\bdy\Omega$. We say that $\Omega$ is \emph{visible around $p$} (resp., $\Omega$ is 
\emph{weakly visible around $p$}) if there exists a holomorphic chart $(U,\varphi)$ of $X$ centered at
$p$ such that $\varphi(U)$ is bounded and $\varphi(U\cap\Omega)\varsubsetneq\Cn$ is a domain, and such
that $\bdy(\varphi(U\cap\Omega))$ is visible (resp., weakly visible). Let $\bdy_V\Omega$ denote the set
of all points $p\in\bdy\Omega$ such that $\Omega$ is visible around $p$ and $\bdy_{WV}\Omega$ denote the set of all points 
$p\in\bdy\Omega$ such that $\Omega$ is weakly visible around $p$.
\end{definition}

Superficially, Definition~\ref{D:visible-around-p} seems to be no different from
Definition~\ref{D:locally-visible}. However, the formulation of the former definition in terms of
holomorphic charts allows us to construct large families of domains $\Omega\varsubsetneq X$, $X$ 
\emph{different} from $\Cn$, having the visibility property by appealing to the many examples one knows
in $\Cn$. This point will be clearer after we state our next theorem.

\begin{theorem}\label{T:visible-around-p-totally-disconnected}
 Let $X$ be any one of the following:
\begin{itemize}[leftmargin=25pt]
  \item[$(a)$] A Kobayashi hyperbolic complex manifold,
  \item[$(b)$] A Stein manifold,
  \item[$(c)$] A holomorphic fiber bundle $\pi: X\to Y$ such that
  $Y$ is either of $(a)$ or $(b)$ and the fibers are Kobayashi hyperbolic,
  \item[$(d)$] A holomorphic covering space $\pi: X\to Y$ such that $Y$ is either of $(a)$ or $(b)$.
\end{itemize}
Let $\OM\varsubsetneq X$ be a relatively compact domain.
\begin{itemize}[leftmargin=27pt]
 \item[$(i)$] Suppose the set $\bdy\Omega\setminus\bdy_V\Omega$ is totally disconnected. Then,
 $\bdy\Omega$ is visible.
 \item[$(ii)$] Suppose the set $\bdy\Omega\setminus\bdy_{WV}\Omega$ is totally disconnected. Then,
  $\bdy\Omega$ is weakly visible.
\end{itemize}   
\end{theorem}

The above result can be interpreted as follows: if, around ``most'' points $p\in\bdy\Omega$, $\Omega$
(locally) looks like a visibility domain in $\C^{\dim(X)}$\,---\,and, as stated above, 
we have \textbf{many} ways of identifying visibility in $\Cn$, thanks to 
Bharali--Zimmer \cite{bharalizimmer:gdwnv17, bharalizimmer:gdwnv23}, Bharali--Maitra
\cite{bharalimaitra:awnovfoewt21}, Bracci \emph{et al.} \cite{braccinikolovthomas:vkgcdrp22}\,---\,then
$\bdy\Omega$ is visible.
\smallskip

Now, we will present a result that is a method
for obtaining new domains with visible boundaries from old domains with the former property.

\begin{theorem}\label{T:domain-totally-discoonected-visible}
Let $X$ be a complex manifold of dimension $n$ and $\Omega\varsubsetneq X$ be a Kobayashi hyperbolic
domain. Let $\bdy\Omega$ be visible. Suppose $S\varsubsetneq X$ is a closed, totally disconnected set
such that $\Omega\setminus S$ is connected, and
such that $\mathcal{H}^{2n-2}(S)=0$. Then,
$\bdy(\Omega\setminus S)$ is visible.    
\end{theorem}
Here, $\mathcal{H}^\alpha$ denotes the $\alpha$-dimensional Hausdorff measure for $\alpha\geq 0$.
\smallskip

To state our next result, we need a definition.

\begin{definition}\label{D:jordan-curve}
Let $\Omega$ be a domain in $\C$. Let $\widehat{\C}$ be the one-point compactification
of $\C$ and $\bdy_\infty\OM$ be the
boundary of $\OM$ viewed as a domain in $\widehat\C$.
\begin{itemize}[leftmargin=25pt]
 \item[$(a)$] We say that $\bdy_\infty\Omega$ is a \emph{Jordan
 curve} if there exists a homeomorphism between $S^1$ and $\bdy_\infty\Omega$.
 \item[$(b)$] We say that $\Omega$ is a \emph{Jordan domain} if $\bdy_\infty\Omega$ is a Jordan curve.
 \end{itemize}
\end{definition}

We now present a new class of planar domains with the visibility property. To the best of my
knowledge, there is no version of the following result in the literature for \textbf{unbounded} domains. 

\begin{proposition}\label{P:simply-connected}
Let $\Omega$ be a simply connected domain in $\C$ such that $\Omega$ is a Jordan domain. Then,
$\bdy\Omega$ is visible.    
\end{proposition}

Experts will be familiar with the notion of ``local Goldilocks points'' introduced in
\cite{bharalizimmer:gdwnv23} providing a sufficient condition for visibility for domains in $\Cn$. Is
there an analogue of this in the more general setting of this paper? The answer to this is ``Yes," but
since formulating this notion in a coordinate-independent way requires some work, we defer this discussion
to Section~\ref{S:sufficint-condition}.
\smallskip

\subsection{Compactifications, visibility domains, and a Wolff--Denjoy-type theorem}
\label{SS:visi-compacti-WD}
The main notions studied in this paper\,---\,i.e., visibility and weak visibility of
$\bdy\Omega$\,---\,formally differ from the notion alluded to at the top of this section and which
motivates this work: that of (weak)
visibility domains with respect to $K_{\Omega}$ for domains
$\Omega\varsubsetneq \Cn$. We refer readers to \cite[Definition~1.1]{bharalimaitra:awnovfoewt21} for
the definition of the latter notion when $\Omega$ is bounded and to
\cite[Definition~1.1]{bharalizimmer:gdwnv23} when $\Omega$ is unbounded. The question arises: \emph{are
the two notions equivalent?} When $\Omega\varsubsetneq \Cn$ is bounded, it is elementary to see that
$\bdy\Omega$ is visible (resp., weakly visible) in the
sense of Definition~\ref{D:visible} if and only if $\Omega$ is a visibility domain (resp., a weak
visibility domain) in the sense of \cite{bharalimaitra:awnovfoewt21, bharalizimmer:gdwnv23}.
However, for $\Omega\varsubsetneq \Cn$
unbounded, the definition of a (weak) visibility domain involves the end compactification of 
$\overline{\Omega}$, and its
relationship with Definition~\ref{D:visible} is
\emph{a priori} unclear. The reasons that the end compactification of $\overline{\Omega}$ forms a part of
\cite[Definition~1.1]{bharalizimmer:gdwnv23} are:
\begin{itemize}[leftmargin=25pt]
 \item[$(a)$] For many of the applications of visibility shown in the works cited above, an
 appropriate compactification of $\Omega$ is needed, when $\Omega$ is unbounded, for proofs to be
 valid\,---\,which is an issue we too must address in this paper. 
 
 \item[$(b)$] The end compactification of $\overline\Omega$ (see \cite[Section~4]{bharalizimmer:gdwnv23}
 for a definition) suggests itself as a natural choice of a compactification because it arises naturally
 in results involving the Kobayashi distance; see, for instance,
 \cite{braccigaussierzimmer:heqicdCdit21}. 
\end{itemize}
In the more general setting of this paper, we formalise the need suggested in $(a)$ above as
follows:

\begin{definition}\label{D:admi-comp}
Let $X$ be a complex manifold and let $\Omega\varsubsetneq X$ be a Kobayashi hyperbolic domain.
We say that $\Omc$ is an \emph{admissible compactification of $\,\overline\Omega$} if $\Omc$ is a
Hausdorff, sequentially compact topological space that admits a map $\emb:\overline\Omega\lrarw\Omc$
that is a homeomorphism onto its image, such that $\emb(\overline\Omega)$ is open and dense in $\Omc$,
and such that $\Omc\setminus \emb(\overline{\Omega})$ is totally disconnected.    
\end{definition}

In the present setting, no preferred compactification suggests itself. Hence the need for the above
definition. In fact, for $\Omega\varsubsetneq \C$, it turns out to be more natural to take $\Omc$ to be
the closure of $\Omega$ in $\widehat{\C}$: the one-point compactification of $\C$. This will be seen in
the results of a forthcoming work (and is hinted at by the proof of Proposition~\ref{P:simply-connected}).
\smallskip

We can now propose a definition for a visibility domain. In the definition below, we will commit the
following abuse of notation: for any point $x\in\overline\Omega$, we shall write $\emb(x)$ simply as $x$.

\begin{definition}\label{D:visibility-domain}
Let $X$ be a complex manifold, $\Omega\varsubsetneq X$ be a domain (not necessarily relatively
compact), and let $\Omc$ be
some admissible compactification of $\overline\Omega$. We say that $\Omega$ is a \emph{visibility
domain subordinate to
$\Omc$} if $\Omega$ is Kobayashi hyperbolic and has the following property:
\begin{itemize}
 \item[$(*)$] Given $\lambda\geq 1$, $\kappa\geq 0$, distinct points $p,q\in\bdy_\infty\Omega$, and
 $\Omc$-open neighbourhoods
 $V_p$ of $p$ and $V_q$ of $q$ whose closures (in $\Omc$) are disjoint, there exists a compact set
 $K\subset\Omega$ such that the image of each $(\lambda,\kappa)$-almost-geodesic 
 $\sigma:[0,T]\lrarw\Omega$ with
 $\sigma(0)\in\emb^{-1} (V_p\setminus\bdy_\infty\Omega)$
 and $\sigma(T)\in\emb^{-1}(V_q\setminus\bdy_\infty\Omega)$ intersects $K$.
\end{itemize}
We say that $\Omega$ is a \emph{weak visibility domain subordinate to
$\Omc$} if the property $(*)$ holds true \textbf{only} for $\lambda=1$.
Here, $\emb$ is the embedding of $\overline\Omega\hookrightarrow\Omc$ that defines $\Omc$
and $\bdy_\infty\Omega := \Omc\setminus \emb(\Omega)$.
\end{definition}

\begin{comment}
To return to a discussion started above: it turns out that (even when $\Omega$ is \textbf{not} relatively
compact) $\bdy\Omega$ is visible (resp., weakly visible) in the sense of Definition~\ref{D:visible} if
and only if $\Omega$ is a visibility domain (resp., a weak visibility domain) subordinate to $\Omc$ for
any admissible compactfication of $\overline{\Omega}$ (see Section~\ref{S:compacti-visi}).
\end{comment}

The question raised above is answered as follows:

\begin{theorem}[see Corollaries~\ref{C:visible-visibility-domain},
\ref{C:visible-weak-visibility-domain}]\label{T:equivalence}
Let $X$ be a complex manifold and $\Omega\varsubsetneq X$ a Kobayashi hyperbolic domain.
Then, $\bdy\Omega$ is visible (resp., weakly visible) if and only if $\Omega$ is a visibility domain
(resp., weak visibility domain) subordinate to any admissible compactification of $\overline\Omega$.
\end{theorem}

We will not overtly focus on the property
in Definition~\ref{D:visibility-domain}.
In this paper, the latter property:
\begin{itemize}
 \item is used to give a unified discussion\,---\,irrespective of whether or not $\Omega\varsubsetneq X$
 is relatively compact\,---\,of the connection between visibility of $\bdy\Omega$ and the local
 Goldilocks property that was alluded to above, and
 \item plays an absolutely vital role in establishing a Wolff--Denjoy-type theorem in the very general
 setting of this paper. 
\end{itemize}

One of the goals of this paper is to show how ubiquitous the Wolff--Denjoy phenomenon is. The
Wolff--Denjoy theorem (see Result~\ref{R:WD} below) has been generalised in many ways in the literature
but, till recently, these
have featured convex domains. Since the history of
generalisations of the Wolff--Denjoy theorem is extensive, and their contrast with our Wolff--Denjoy-type
theorem (see Theorem~\ref{T:wolff-denjoy}) is significant, we shall introduce and discuss
it in Section~\ref{S:application}.

\subsection{Connections with Gromov hyperbolicity}\label{SS:Gromov-connection}
We will assume that the reader is familiar with Gromov hyperbolic metric
spaces. Now, length metric spaces that are not geodesic spaces (and hence, by the Hopf--Rinow Theorem, are
not Cauchy-complete) are less frequently encountered in geometry. But, as we have emphasised above,
we do not assume in this paper that $(\Omega, K_{\Omega})$\,---\,for $\Omega$ as above\,---\,is
Cauchy-complete. The definition of Gromov hyperbolicity of $(\Omega, K_{\Omega})$ in the latter setting
might be useful. We refer the reader to the definition in Section~\ref{SS:Gromov}. Every Gromov
hyperbolic metric space $(X,d)$ has an abstract boundary called the \emph{Gromov boundary}, denoted by
$\bdy_G X$. We define \emph{Gromov bordification} to be the \textbf{set} $\overline X^G:=X\cup\bdy_G X$
equipped with a topology with respect to which the inclusion $X\hookrightarrow \overline X^G$ is a homeomorphism.
This is called the \emph{Gromov topology}, which we will discuss in Section~\ref{SS:Gromov}.
Here, we introduce the notation ``$\xrightarrow{G} \xi$'' to denote the approach to $\xi \in\bdy_G X$
with respect to the Gromov topology.
\smallskip

When a length metric space $(X, d)$ is Cauchy-complete and Gromov hyperbolic, then it is
well known that
$\bdy_G X$ is visible to geodesic lines (see
\cite[Part~III, Lemma~3.2]{bridsonhaefliger:msnpc99}). This property of $\bdy_G X$ is key to
a range of applications of Gromov hyperbolicity. 
We now specialise to the metric space $(\Omega, K_{\Omega})$, $\Omega$ is as in the theorems above.
This, in view of
Proposition~\ref{P:almost-geodesic}, raises the natural question: \emph{when $(\Omega, K_{\Omega})$ is
\textbf{not} Cauchy-complete, then is $\bdy_G\Omega$ visible to almost-geodesics?} An answer to this
question would, among other things, give us necessary conditions for when $\overline{\Omega}$ embeds into 
$\overline\Omega^G$ (which would open up many applications in the present setting). Firstly, we must
clarify what is meant by the above question. To this end, we need a definition.

\begin{definition}\label{D:gromov-boundary-visible}
Let $X$ be a complex manifold and $\Omega\varsubsetneq X$ be a Kobayashi hyperbolic domain (that is not
necessarily Kobayashi complete) such that $(\Omega,K_\Omega)$ is Gromov hyperbolic. We say that
$\bdy_G\Omega$ is \emph{visible to almost-geodesics} if for any $\xi, \eta\in\bdy_G\Omega$,
$\xi\neq \eta$, given any $\kappa\geq 0$ and sequences $(x_\nu)_{\nu\geq 1}\subset\Omega$ and
$(y_\nu)_{\nu\geq 1}\subset\Omega$ such that $x_\nu\xrightarrow{G} \xi$, and $y_\nu\xrightarrow{G} \eta$,
there exists a compact set $K\subset \Omega$ such that for each $\nu$, the image of each 
$(1,\kappa)$-almost-geodesic joining $x_\nu$ and $y_\nu$ intersects $K$.    
\end{definition}

It turns out that the answer to the above question is in the affirmative in many cases.

\begin{theorem}\label{T:gromov-boundary-visible}
Let $X$ be a complex manifold and $\Omega\varsubsetneq X$ be a Kobayashi hyperbolic domain such that
$(\Omega,K_\Omega)$ is Gromov hyperbolic. Suppose $\overline\Omega^G$ is locally compact. Then,
$\bdy_G\Omega$ is visible to almost-geodesics.  
\end{theorem}

The question now arises: \emph{how natural is the assumption in Theorem~\ref{T:gromov-boundary-visible} that
$\overline\Omega^G$ be locally compact
when $(\Omega,K_\Omega)$ is incomplete?} The final result of this section gives a construction of a
large collection of domains that satisfy the conditions in Theorem~\ref{T:gromov-boundary-visible}.

\begin{proposition}\label{P:gromov-visible-example}
Let $D\varsubsetneq\Cn$ be a Kobayashi hyperbolic domain such that 
\begin{itemize}[leftmargin=25pt]
  \item[$(a)$]$(D,K_D)$ is Gromov hyperbolic,
  \item[$(b)$]$id_D$ extends to a homeomorphism between $\overline D$ and $\overline D^G$.
\end{itemize}
Let $K\subset D$ be a compact set such that $\mathcal{H}^{2n-2}(K)=0$. Write
$\Omega:=D\setminus K$. Then,
\begin{itemize}[leftmargin=27pt]
  \item[$(i)$]$\Omega$ is Kobayashi hyperbolic and $(\Omega,K_\Omega)$ is Gromov hyperbolic, 
  \item[$(ii)$]$\overline\Omega^G$ is homeomorphic to $\overline D\setminus K$.
\end{itemize}
In particular $\overline\Omega^G$ is non-compact and locally compact.
\end{proposition}

The conditions $(a)$ and $(b)$ are observed rather frequently in this field. For instance, these
conditions are satisfied by bounded strongly pseudoconvex domains \cite{baloghbonk:GhKmspd00} and
by $\smoo^\infty$-smoothly bounded convex domains of finite type \cite{zimmer:GhKmcdft16}.
\medskip

\section{Metrical preliminaries and almost-geodesics}\label{S:metrical}
In this section, we elaborate upon certain metrical statements made in Section~\ref{S:intro} for which the
associated definitions and explanations had been deferred. 

\subsection{$\boldsymbol{(\lambda, \kappa)}$-almost-geodesics}\label{SS:AG}
We present a couple of definitions.

\begin{definition}\label{D:abs-cont}
Let $X$ be a complex manifold of dimension $n$. A path $\sigma: I\lrarw X$, where $I\subseteq \R$ is an
interval, is said to be \emph{locally absolutely 
continuous} if for each $t_0\in I$ and each holomorphic chart $(U,\varphi)$ around $\sigma(t_0)$, with
$I(\varphi,t_0)$
denoting any closed and bounded interval in $\sigma^{-1}(U)$ containing $t_0$,
$\varphi\circ\sigma{\mid}_{I(\varphi,t_0)}$ is
absolutely continuous as a path in $\R^{2n}$.
The path $\sigma: I\lrarw X$ is said to be
\emph{absolutely continuous} if $I\varsubsetneq \R$ is a closed and bounded interval and $\sigma$ is
locally absolutely continuous.
\end{definition}

\begin{remark}\label{rem:abs-cont}
From Definition~\ref{D:abs-cont} and the fact that every {locally} absolutely continuous 
path in $\R^{2n}$ (i.e., a path whose restriction to each closed and bounded interval is 
absolutely continuous) is almost-everywhere differentiable, it follows that the same holds true for the
paths defined in Definition~\ref{D:abs-cont}. If $\sigma$ is such a path and $t\in I$ is such that
$D\sigma(t)$ exists, there is a canonical identification of the vector $D\sigma(t)1\in T_{\sigma(t)}X$ with a
vector in $T{^{(1,0)}_{\sigma(t)}}X$. In this paper, we shall denote the latter by $\sigma'(t)$.
If $I\varsubsetneq \R$ is a closed and bounded interval, then $\sigma: I\lrarw \R^{2n}$ is
absolutely continuous if and only if it is locally absolutely continuous; this is the motivation for our
definition of absolute continuity of $\sigma$ in Definition~\ref{D:abs-cont}.
\end{remark}

\begin{definition}\label{D:almost-geodesic}
Let $X$ be a complex manifold of dimension $n$ and let $\Omega\subseteq X$ be a Kobayashi hyperbolic
domain. Let
$I\subseteq \R$ be an interval. For
$\lambda\geq 1$ and $\kappa\geq 0$, a curve $\sigma: I\lrarw \Omega$ is said to be a 
\emph{$(\lambda, \kappa)$-almost-geodesic} if 
\begin{itemize}[leftmargin=25pt]
 \item[$(a)$] for all $s,t\in I$
 \[
   \frac{1}{\lambda}|s-t|-\kappa\leq K_\Omega (\sigma(s),\sigma(t))\leq \lambda|s-t|+\kappa,
 \]
 \item[$(b)$] $\sigma$ is locally absolutely continuous (whence $\sigma'(t)$ exists for almost every $t\in I$),
 and for almost every $t\in I$,
 $k_\Omega(\sigma(t);\sigma'(t))\leq\lambda$ (here $k_\Omega$ denotes the Kobayashi pseudometric on $\Omega$).
\end{itemize}
\end{definition}

\subsection{General properties of Kobayashi hyperbolic domains}\label{SS:property-kobayashi}
With $X$ and $\Omega$ as above and $\sigma: [a,b]\lrarw\Omega$ an absolutely continuous path in
$\Omega$, we
introduce \emph{the length of $\sigma$ relative to the Kobayashi pseudometric}, denoted by
$l_\Omega{(\sigma})$. For
this, we first observe that the Kobayashi pseudometric $k_\Omega$ is an upper semicontinuous function on
$T^{(1,0)}\Omega:=T^{(1,0)}X{\mid}_{\Omega}$. Hence, from
Remark~\ref{rem:abs-cont}, $k_\Omega(\sigma(t);\sigma'(t))$ is integrable. Thus, the integral below
makes sense and we define
\[
  l_\Omega(\sigma):=\int_{a}^{b} k_\Omega(\sigma(t);\sigma'(t))\,dt.
\]

For a Kobayashi hyperbolic domain $\Omega\subseteq X$, the Kobayashi pseudometric $k_\Omega$ has the
following properties:

\begin{result}\label{R:kob-metr-prop}
Let $X$ be a complex manifold of dimension $n$ and $\Omega\subseteq X$ be a Kobayashi hyperbolic domain.
\begin{itemize}[leftmargin=25pt]
 \item[$(a)$] \cite[Theorem~1]{royden:rkm71} For any $z, w\in\Omega$ we have
 \[
   \qquad\quad K_\Omega(z,w) = \inf\{l_\Omega(\sigma) : \sigma: [a,b]\lrarw\Omega 
   \text{ is piecewise $\smoo^1$ with $ \sigma(a)=z$ and $\sigma(b)=w$}\}.
 \]
 \item[$(b)$] \cite[Theorem~3.1]{venturini:pprcm89} For any $z, w$ we have
 \begin{multline*}
 \qquad\qquad K_\Omega(z,w) =\inf\{l_\Omega(\sigma): \sigma: [a,b]\lrarw\Omega \\
 \text{is absolutely continuous with $\sigma(a)=z$ and $\sigma(b)=w$}\}.
 \end{multline*}
\end{itemize}
\end{result}

\begin{result}[paraphrasing {\cite[Theorem~2]{royden:rkm71}}]\label{R:kob-hyp-cond}
Let $X$ be a complex manifold of dimension $n$ and $\Omega\subseteq X$ be a
domain. Fix a Hermitian metric $h$ on $X$ and let $d_h$ denote the distance induced by $h$.
Then, $\Omega$ is Kobayashi hyperbolic if and only if for each $x\in\Omega$, there exist constants
$c_{x}>0$, $r_{x}>0$
such that $k_\Omega(y;v)\geq c_{x}h_y(v)$ for every
$v\in T{^{(1,0)}_{y}} \Omega$ and every $y\in B_{d_h}(x,r_x)$.  
\end{result}

\begin{remark}\label{rem:open-d-h-ball-notation}
In the above statement, $B_{d_h}(x,r)$ denotes the open $d_h$-ball centered at $x$ and having radius
$r$. In the above,
and elsewhere in this paper, we abbreviate $\sqrt{h|_x(v,v)}$, $v\in T{^{(1,0)}_{x}}X$, as $h_{x}(v)$.
The above result implies that when $\Omega$ is Kobayashi hyperbolic, the Kobayashi pseudometric
$k_\Omega$ is a (Finsler) metric.  
\end{remark}

Now, we will prove that in the metric space $(\Omega, K_\Omega)$ (where $\Omega$ is a Kobayashi hyperbolic domain in a
complex manifold $X$),
any two points in $\Omega$ can be joined by a $(1,\kappa)$-almost-geodesic. To prove this result, we
need a technical
lemma.

\begin{lemma}\label{L:to-prove-almost-geodesic}
Let $X$ be a complex manifold and $\Omega\subseteq X$ be a Kobayashi hyperbolic domain. Let
$\sigma: [a,b]\lrarw \Omega$ be an absolutely continuous path. If 
\[
  l_\Omega(\sigma)\leq K_\Omega(\sigma(a),\sigma(b))+\varepsilon,
\]
for some $\varepsilon>0$, then, whenever $a\leq a'\leq b'\leq b$, we have
\[
  l_\Omega(\sigma{\mid}_{[a',b']})\leq K_\Omega(\sigma(a'),\sigma(b'))+\varepsilon.
\]
\end{lemma}

In view of part $(b)$ of Result~\ref{R:kob-metr-prop}, the above lemma follows from the definition of
$l_\Omega(\sigma)$
and the triangle inequality for the Kobayashi distance on $\Omega$.

\begin{proposition}\label{P:almost-geodesic}
Let $X$ be a complex manifold and $\Omega\subseteq X$ be a Kobayashi hyperbolic domain in $X$. For any
$\kappa>0$ and $x, y\in\Omega$, there exists a $(1,\kappa)$-almost-geodesic $\sigma: [a,b]\lrarw
\Omega$ with $\sigma(a)=x$
and $\sigma(b)=y$.   
\end{proposition}

\begin{proof}
Fix a Hermitian metric $h$ on $X$ and let $d_h$ be the distance induced by $h$ on $X$. It is a 
classical result that
\[
  Top(\Omega)=Top(d_h{\mid}_{\Omega\times\Omega}),
\]
where $Top(\Omega)$ denotes the manifold topology on $\Omega$. Now, from part
$(a)$ of Result~\ref{R:kob-metr-prop}, there exists a piecewise $\smoo^1$ curve $\gamma:
[0,T]\lrarw\Omega$ such that $\gamma(0)=x$,
$\gamma(T)=y$, and
\begin{align}\label{E:L_Omega_gamma}
 l_\Omega(\gamma)<K_\Omega(x,y)+\kappa.
\end{align}
Since $\gamma([0,T])$ is compact, it can be covered by finitely many charts on $\Omega$.
Also the Kobayashi metric
$k_\Omega$ is upper semicontinuous on $T^{(1,0)}\Omega$. Therefore, perturbing $\gamma$ if needed, 
we can assume that
$\gamma$ is $\smoo^1$ smooth and that $h_{\gamma(t)}({\gamma}'(t))\neq 0$ for all $t\in [0,T]$. Now, we define
the function $f: [0,T]\lrarw \R$ by
\[
  f(t)=\int_{0}^{t}k_\Omega(\gamma(s);\gamma'(s))\,ds.
\]
Since $k_\Omega$ is upper semicontinuous on $T^{(1,0)}\Omega$, there exists a constant $C_1>0$ such
that 
\begin{equation}\label{E:right-constant-kobayashi-metric}
  k_\Omega(\gamma(t);\gamma'(t))\leq C_1
  h_{\gamma(t)}(\gamma'(t)) \quad \forall t\in [0,T].
\end{equation}
Since $\Omega$ is Kobayashi hyperbolic, from Result~\ref{R:kob-hyp-cond} it follows that for each 
$t\in [0,T]$, there exist constants
$r_t>0$, $\wt{C}_{t}>0$ such that for all $y\in B_{d_h}{(\gamma(t),r_t)}$ and for all
$v\in T{^{(1,0)}_{y}} \Omega$
\[
  k_\Omega(y;v)\geq \wt{C}_{t} h_y (v).
\]
Since $\gamma([0,T])$ is compact, there exist $t_1,...,t_N\in[0,T]$ such that 
$B_{d_h}(\gamma(t_1),r_{t_1}),...,B_{d_h}
(\gamma(t_N),r_{t_N})$ cover $\gamma([0,T])$. Therefore, there exists a constant
$C_2>0$ such that, taking \eqref{E:right-constant-kobayashi-metric} into consideration, we have
\begin{align*}
 {C_2}h_{\gamma(t)}(\gamma'(t))\leq k_\Omega(\gamma(t);\gamma'(t))\leq C_1 h_{\gamma(t)}(\gamma'(t)),   
\end{align*}
for all $t\in [0,T]$.
By assumption, $h_{\gamma(t)}({\gamma}'(t))\neq 0$ for all $t\in [0,T]$. As $h$ is continuous,
there exists a constant $\alpha>1$ such that
\begin{align*}
 (1/\alpha)\leq C_2 h_{\gamma(t)}(\gamma'(t))&\leq k_\Omega(\gamma(t);\gamma'(t))\leq C_1 h_{\gamma(t)}
 (\gamma'(t))\leq \alpha \notag \\ \implies (1/\alpha)|t_1-t_2|&\leq|f(t_1)-f(t_2)|\leq\alpha |t_1-t_2|,    
\end{align*}
for all $t_1,t_2\in [0,T]$. Hence, $f$ is bi-Lipschitz and strictly
increasing function. Therefore, we can define the function 
$g=f^{-1}: [0,l_\Omega(\gamma)]\lrarw[0,T]$.
\medskip

\noindent{{\textbf{Claim.}}}
\emph{$\wt{\gamma}:=\gamma\circ g$ is a $(1,\kappa)$-almost-geodesic joining $x$ and $y$.}
\smallskip

\noindent{\emph{Proof of Claim.}
Since $f$ is bi-Lipschitz, $g$ is bi-Lipschitz. Now, fix $t_0\in [0,l_\Omega(\gamma)]$, fix a
holomorphic chart $(U,\varphi)$ around
$\wt{\gamma}(t_0)$, and let $I(\varphi,t_0)$ be a closed interval of $\wt{\gamma}^{-1}
(U)$ containing $t_0$. So, $g(I(\varphi,t_0))$ is a closed interval. Since $\varphi$ is smooth, and
$\gamma$ is $\smoo^1$ smooth on $[0,T]$, $\varphi\circ{\gamma}{\mid}_{g(I(\varphi,t_0))}$ is 
Lipschitz. Hence, $\varphi\circ\wt{\gamma}{\mid}_{I(\varphi,t_0)}$ is Lipschitz.
Therefore, $\varphi\circ\wt{\gamma}{\mid}_{I(\varphi,t_0)}$ is an
absolutely continuous path. So, by Definition~\ref{D:abs-cont}, $\wt{\gamma}$ is absolutely continuous.
\smallskip 

Since $k_\Omega(\gamma(s);\gamma'(s))$ is upper semi-continuous (and hence measurable) and bounded on $[0,T]$, it is integrable on $[0,T]$. By Lebesgue
differentiation
theorem, there exists a set $E\subseteq[0,T]$ of full measure such that $f$ is differentiable and 
$f'(s)=k_\Omega(\gamma(s);\gamma'(s))$ for each $s\in E$.
Since $g$ is bi-Lipschitz, $g^{-1}(E)\subseteq[0,l_\Omega(\gamma)]$ has full measure, hence on $g^{-1}(E)$, $g$ is
differentiable, and for all $s\in g^{-1}(E)$:
\[
  g'(s)= 1/f'(g(s)).
\]
At this stage we are reduced to examining functions 
defined on intervals in $\R$. No charts of $X$ are now involved. Thus, in view of
Lemma~\ref{L:to-prove-almost-geodesic}, the proof that $\wt{\gamma}$ has properties $(a)$ and $(b)$
given in Definition~\ref{D:abs-cont} now proceeds exactly as in the proof of 
\cite[Proposition~4.4]{bharalizimmer:gdwnv17}. \hfill $\btl$} 
\smallskip

Since the above claim is true, the result follows.
\end{proof}

We conclude this section with a technical lemma. Its proof is trivial, despite its
technical hypothesis, and we state it here because we shall invoke it several times in the later
sections.

\begin{lemma}\label{L:3_nbhd}
Let $X$ be a complex manifold and $\Omega\varsubsetneq X$ be a Kobayashi hyperbolic domain in $X$. Let
$h$ be a Hermitian
metric on $X$ and let $(K_{\nu})_{\nu\geq 1}$ be an exhaustion by compacts of $\Omega$. Fix
$\lambda\geq 1$,
$\kappa\geq 0$, and $p\in \bdy\Omega$. Let $U$ and $W$ be neighbourhoods of $p$ in $X$ such that $U$
is relatively
compact and $W\Subset U$. Let $(x_{\nu})_{\nu\geq 1}\subset W\cap \Omega$ and, for each $\nu$, let
$\gamma_{\nu}:
[a_{\nu}, b_{\nu}]\lrarw \Omega$ be a $(\lambda, \kappa)$-almost-geodesic such that 
$\gamma_{\nu}(a_{\nu}) = x_{\nu}$ and
such that
\[
  \gamma_{\nu}([a_{\nu}, b_{\nu}]) \subset \Omega\setminus K_{\nu} \quad\text{and}
  \quad \gamma_{\nu}(b_{\nu})\notin U.
\]
Given a neighbourhood $V$ of $p$ in $X$ such that $W\Subset V\Subset U$, let 
$r_{\nu}\in (a_{\nu}, b_{\nu})$ be such that
\[
  y'_{\nu} := \gamma_{\nu}(r_{\nu})\in V\setminus \overline{W} \quad\text{and}
  \quad \gamma_{\nu}([a_{\nu}, r_{\nu}]) \subset U \quad \forall \nu\geq 1.
\]
Write $\wt{\gamma}_{\nu} := \gamma_{\nu}|_{[a_{\nu}, r_{\nu}]}$. Then, there exist a point
$p'\in \overline{W}\cap\bdy\Omega$ and a subsequence $(\wt{\gamma}_{\nu_k})_{k\geq 1}$
such that $x_{\nu_k}\to p'$ and
\[
  \lim_{k\to\infty}\sup_{z\in{\sf image}(\wt\gamma_{\nu_k})}d_h(z,\bdy\Omega)=0.
\]
\end{lemma}

\begin{remark}
In the above statement, $d_h$ denotes the distance induced by $h$ on $X$. The existence of a
neighbourhood $V$
and numbers $r_{\nu}\in (a_{\nu}, b_{\nu})$, $\nu\geq 1$, with the properties stated in the hypothesis
above is
guaranteed by the condition $\gamma_{\nu}(b_{\nu})\notin U $ for each $\nu\geq 1$.
\end{remark}

The proof of Lemma~\ref{L:3_nbhd} is a consequence of the condition
$\gamma_{\nu}([a_{\nu}, b_{\nu}]) \subset \Omega\setminus K_{\nu}$ for all $\nu\geq 1$ and the fact that
$U$ is relatively compact, and is trivial. Thus, we shall skip the proof.
\smallskip

\subsection{Gromov hyperbolic metric spaces}\label{SS:Gromov}
In Section~\ref{SS:Gromov-connection} we mentioned Gromov hyperbolic metric spaces. Let us begin with the
definition of Gromov hyperbolicity of a metric space $(X, d)$ that does not require $(X,d)$ to be a
geodesic space. But first, for $x,y,o\in X$, we define the \emph{Gromov product} $(x|y)_o$
by $(x|y)_o :=(d(x,o)+d(y,o)-d(x,y))/2$.

\begin{definition}\label{D:Gromov-hyperbolic}
Let $(X,d)$ be a metric space. We say that
$(X, d)$ is \emph{Gromov hyperbolic} if there exists $\delta\geq 0$ such that for any four
points $x, y, z, o\in X$,
\begin{align}\label{E:gromov-hyperbolic}
 \min\big\{(x|y)_o ,(y|z)_o \big\} \leq (x|z)_o+\delta.
\end{align}
\end{definition}

Here, we will present some
more notations and a useful result to prove Theorem~\ref{T:gromov-boundary-visible}. These are taken
from \cite[Section~3.4]{dassimmonsurbanski:gdghms17}.

\begin{definition}\label{D:gromov-product}
 Let $(X,d)$ be a metric space.
\begin{itemize}[leftmargin=25pt]
  \item[$(a)$]We say that a sequence $(x_\nu)_{\nu\geq 1}\subset X$ is a \emph{Gromov sequence} if for
  some (hence any) $o\in X$,
  \[
    \lim_{\nu,l\to\infty}(x_l|x_\nu)_o=+\infty.
  \]
  \item[$(b)$]Let $(X,d)$ be Gromov hyperbolic. We say that two Gromov sequences $(x_\nu)_{\nu\geq 1}$
  and $(y_\nu)_{\nu\geq 1}$ are \emph{equivalent} if for some (and hence any) $o\in X$,
  \[
    \lim_{\nu,l\to\infty}(x_l|y_\nu)_o=+\infty.
  \]
\end{itemize}
If $(X,d)$ is Gromov hyperbolic, it follows that the above relation (between two sequences of $X$)
is an equivalence relation.
\end{definition}

We shall denote the equivalence class of Gromov sequences $(x_\nu)_{\nu\geq 1}$ by $[x_\nu]$. We define the 
\emph{Gromov boundary} of $X$ to be the set of all equivalence classes of Gromov sequences in $X$. We
shall denote the Gromov boundary of $X$ by $\bdy_G X$. Recall that the \emph{Gromov bordification} of $X$,
denoted by $\overline X^G$, as a \textbf{set} is $\overline X^G := X\cup \bdy_G X$. Note that $\overline X^G$ is not compact when $X$ is not complete. We will not
elaborate upon the construction of the topology on $\overline X^G$; we refer the reader to 
\cite[Section~3.4]{dassimmonsurbanski:gdghms17} for
details about Gromov topology. Instead, we present the following result which is relevant for our paper
and also gives some idea about the Gromov topology.  

\begin{result}\label{R:gromov-topology-property}
 Let $(X,d)$ be a Gromov hyperbolic metric space. Then,
\begin{itemize}[leftmargin=25pt]
  \item[$(a)$] The Gromov topology depends only on $d$ and the inclusion $X\hookrightarrow \overline X^G$
  ($X$ being equipped with the metric topology) is homeomorphic.
  \item[$(b)$]A sequence $(x_\nu)_{\nu\geq 1}\subset X$ converges to a point $\xi \in\bdy_G X$ if and
  only if $(x_\nu)_{\nu\geq 1}$ is a Gromov sequence and $[x_\nu]= \xi$. 
\end{itemize}
\end{result}

\smallskip

\section{The proofs of Theorems~\ref{T:local-global-visibility} and~\ref{T:local-global-weak-visibility}  }\label{S:proof-local-golbal-visibility}
To prove the above-mentioned theorems, we will need the following result:

\begin{result}[paraphrasing {\cite[Lemma~2 of \S4]{royden:rkm71}}]\label{R:royden-localization}
Let $X$ be a complex manifold and $\Omega\subseteq X$ be a Kobayashi hyperbolic domain. Let
$U$
be an open subset of $X$ such that $U\cap\Omega\neq\emptyset$ with $U\cap\Omega$ connected. Then, for every 
$v\in T^{(1,0)}_{x}
(U\cap\Omega)$ and every $x\in U\cap\Omega$, we have
\[
  k_\Omega(x;v)\leq k_{U\cap\Omega}(x;v)\leq \coth{(K_\Omega(x,\Omega\setminus U))}k_\Omega(x;v).
\]
\end{result}

In what follows, if $\sigma:I\lrarw\Omega$, $I\subseteq\R$ an interval, is a curve, then we will denote the image of
$\sigma$ by $\langle\sigma\rangle$.
In the proof that follows, we shall use
the notation introduced in Section~\ref{S:metrical}.

\begin{proof}[The proof of Theorem~\ref{T:local-global-visibility}]
Fix a Hermitian metric $h$ on $X$ and let $d_h$ be the distance induced by $h$ on $X$.  First we will
prove that $(a)$
implies $(b)$. In order to prove that
$\bdy\Omega$ is visible, fix $p,q\in\bdy\Omega$ such that $p\neq q$. Since
$\bdy\Omega$ is locally visible, there exists a neighbourhood $U_p$ of $p$ in $X$ such that
$U_p\cap\Omega$
is a Kobayashi hyperbolic domain and there is a neighbourhood $V_p$ of $p$ in $X$, $V_p\subseteq U_p$,
such that every pair of distinct points
$q_1,q_2\in V_p\cap\bdy\Omega$ satisfies the visibility condition with respect to $K_{U_p\cap\Omega}$.
We also fix $\lambda\geq 1$ and $\kappa\geq 0$. Since $\Omega$ is a
hyperbolically imbedded domain in $X$, shrinking $V_p$ if needed, we may assume without loss of
generality that $V_p$ is relatively compact and that
$K_\Omega(V_p\cap \Omega, \Omega\setminus U_p)=:\delta_p>0$ and that $q\notin\overline{V_p}$.
Therefore, there exists a
neighbourhood $V_q$ of $q$ in $X$ such that $V_q$ is relatively compact, and such that
$\overline{V_p}\cap\overline{V_q}=\emptyset$.  Let $W_p$ be a neighbourhood of $p$ in $X$ such that
$\overline{W_p}\subset V_p$.
For any pair $x,y\in\Omega$ we define the set
\[
  AG_{\Omega}(x,y):=\{\gamma:[0,T]\lrarw\Omega:
  \text{$\gamma$ is a $(\lambda,\kappa)$-almost-geodesic joining $x$ and $y$}\}.
\]
Here, the parameters $\lambda$ and $\kappa$ are those that we fixed above (we do not assume that
$AG_\Omega(x,y)\neq\emptyset$  when $\kappa=0$).
\medskip

\noindent{{\textbf{Claim~1.}}}
\emph{There exists a compact set $K\subset \Omega$ such that for 
every $x\in W_p\cap\Omega$, every
$y\in V_q\cap\Omega$, and every $\gamma\in AG_\Omega(x,y)$, we have $\langle\gamma\rangle\cap K\neq\emptyset$.}
\smallskip

\noindent{\emph{Proof of Claim~1.}
If possible let the claim be false. Let $(K_\nu)_{\nu\geq 1}$ be an exhaustion by
compacts of
$\Omega$. Then, there exist sequences $(x_\nu)_{\nu\geq 1}\subset W_p\cap\Omega$,
$(y_\nu)_{\nu\geq 1}\subset V_q\cap\Omega$, and $(\lambda,\kappa)$-almost-geodesics $\gamma_\nu$ in 
$AG_\Omega(x_\nu,y_\nu)$, for each $\nu\geq 1$, such that $\langle\gamma_\nu\rangle\cap
K_\nu=\emptyset$.
Fix neighbourhoods $W_{p}^1$, $W_{p}^2$ of $p$ such that $W_p\Subset W_p^1\Subset W_p^2\Subset V_p.$
Let $r_\nu>0$ and ${y'_\nu}=\gamma_\nu(r_\nu)$ be such that ${y'_\nu}\in
(W_p^2\setminus\overline{W_p^1})\cap\Omega$, and
such that the image of $\wt{\gamma}_{\nu}:=\gamma_\nu{\mid}_{[0,r_\nu]}$ is contained in $W_p^2$.
By Lemma~\ref{L:3_nbhd}, and as we can relabel the subsequences given by it, there exists a point
$p'\in \overline{W_p}\cap\bdy\Omega$
such that $x_{\nu}\to p'$ and
\begin{align}\label{E:d_h}
 \lim_{\nu\to\infty}\sup_{z\in\langle\wt\gamma_\nu\rangle}d_h(z,\bdy\Omega)=0.
\end{align}
From Result~\ref{R:royden-localization} and the fact that $K_\Omega(z,\Omega\setminus U_p)
\geq \delta_p$ for all $z\in
V_p\cap\Omega$, it follows that for every $v\in T{^{(1,0)}_{z}}
\Omega$ and every $z\in V_p\cap\Omega$
\begin{align*}
 k_\Omega(z;v)\leq k_{U_p\cap\Omega}(z;v)&\leq \coth(K_\Omega(z,\Omega\setminus U_p))k_\Omega(z;v)
 \leq Ak_\Omega(z;v),  
\end{align*}
where $A:=\coth\delta_p (>1)$. Furthermore,
\begin{align*}
 K_\Omega(z,w)\leq K_{U_p\cap\Omega}(z,w)   
\end{align*}
for all $z,w\in V_p\cap\Omega$. Since $\langle\wt{\gamma}_{\nu}\rangle\subset V_p\cap\Omega$ and
$\gamma_\nu$ is a $(\lambda,\kappa)$-almost-geodesic for all
$\nu\geq 1$ and as $A>1$, it follows from the last few inequalities that for all $t,s\in[0,r_\nu]$
\begin{itemize}[leftmargin=27pt]
 \item[$(i)$]
 $k_{U_p\cap\Omega}(\wt{\gamma}_{\nu}(t);{\wt{\gamma}_\nu}'(t))\leq A\lambda$ for a.e. $t\in [0, r_{\nu}]$, 
 \item[$(ii)$]
 ${(A\lambda)^{-1}}|t-s|-A\kappa\leq K_{U_p\cap\Omega}(\wt{\gamma}_{\nu}(t),\wt{\gamma}_{\nu}(s))\leq
 A\lambda|t-s|\leq A\lambda|t-s|+A\kappa$.
\end{itemize}
The second inequality relies on $(i)$ and part $(b)$ of Result~\ref{R:kob-metr-prop}.
Hence, $\wt{\gamma}_{\nu}$ is a $(A\lambda,A\kappa)$-almost-geodesic with respect to
$K_{U_p\cap\Omega}$ joining $x_\nu$ to $y'_\nu$. As $W_p^2$ is relatively compact, there exists a
sequence $(\nu_j)_{j\geq 1}\subset \N$ and a point $q'\in\overline{W_p^2}\cap\bdy\Omega\subseteq
V_p\cap\bdy\Omega$ such that
$\gamma_{\nu_j}(r_{\nu_j})\to q'$. By construction, $p'\neq q'$. Now, since $V_p\cap\bdy\Omega$
satisfies the visibility
condition with respect to
$K_{U_p\cap\Omega}$, there exists a compact set $K\subset U_p\cap\Omega$ such that for every $j\geq
1$ we can find
$z_j\in\langle\wt\gamma_{\nu_j}\rangle$ with $z_j\in K$, which contradicts \eqref{E:d_h}. Hence, the
claim.
\hfill $\btl$}
\smallskip

By this claim, we have established that $\bdy\Omega$ is visible.
\smallskip

Now to prove $(b)$ implies $(a)$, we only need to prove that $\Omega$ is a hyperbolically imbedded
domain. We shall
establish the equivalent property stated in Remark~\ref{rem:hyp-imbed}. Let
$x,y\in\overline\Omega$ such that $x\neq y$. When at least one of $x,y$ is in $\Omega$, this property
follows easily from
the fact that $\Omega$ is Kobyashi hyperbolic, whereby $Top(\Omega)=Top(K_\Omega)$. Thus, consider the
case when both
$x,y\in\bdy\Omega$.
Since $\bdy\Omega$ is visible, there exist neighbourhoods $V_x$ of $x$ and $V_y$ of $y$ in $X$ such that
$\overline{V_x}\cap
\overline{V_y}=\emptyset$ and such that there exists a compact set
$K_0\subset \Omega$
such that the image of each (1,1)-almost-geodesic $\gamma:[0,T]\lrarw\Omega$ with $\gamma(0)\in V_x$
and $\gamma(T)\in
V_y$ intersects $K_0$. 
\medskip

\noindent{{\textbf{Claim~2.}}}
\emph{$K_\Omega(V_x\cap\Omega,V_y\cap\Omega)>0$}
\smallskip

\noindent{\emph{Proof of Claim~2.}
If possible let $K_\Omega(V_x\cap\Omega,V_y\cap\Omega)=0$. Then, there exist sequences
$(z_\nu)_{\nu\geq 
1}\subset V_x\cap\Omega$, $(w_\nu)_{\nu\geq 1}\subset V_y\cap\Omega$ such that
$K_\Omega(z_\nu,w_\nu)\to 0$ as
$\nu\to\infty$. By Proposition~\ref{P:almost-geodesic}, we can find
$\tau_\nu:[0,T_\nu]\lrarw\Omega$, a $(1,1/\nu)$-almost-geodesic such that $\tau(0)=z_\nu$ and
$\tau_\nu(T_\nu)=w_\nu$ for each $\nu\geq 1$. Clearly, each $\tau_\nu$ is also a $(1,1)$-almost-geodesic
with respect to $K_\Omega$. Therefore, by assumption, $\langle\tau_\nu\rangle\cap
K_0\neq\emptyset$ for all $\nu\geq 1$. Let $s_\nu\in[0,T_\nu]$ be such that
$\tau_\nu(s_\nu)=:a_\nu\in\langle\tau_\nu\rangle\cap K_0$ and, by passing to a
subsequence if needed, assume that $a_\nu\to a\in K_0$ as $\nu\to\infty$. Since $\tau_\nu$ is a
$(1,1/\nu)$-almost-geodesic
passing through $a_\nu$, we have
\begin{align*}
 K_\Omega(z_\nu,w_\nu)\geq T_\nu-1/\nu &=\left(T_\nu-s_\nu+1/\nu\right)+
 \left(s_\nu-0+1/\nu\right)-3/\nu \\ 
 &\geq K_\Omega(z_\nu,a_\nu)+K_\Omega(a_\nu,w_\nu)-3/\nu,
\end{align*}
for each $\nu\geq 1$.
Therefore, $K_\Omega(z_\nu,a_\nu)\to 0$ and $K_\Omega(a_\nu,w_\nu)\to 0$ as 
$\nu\to\infty$. Since $\Omega$ is a Kobayashi hyperbolic domain, from the discussion above it follows
that $z_\nu\to a$ and $w_\nu\to a$ as $\nu\to\infty$, which contradicts the fact that
$z_\nu\in V_x\cap\Omega$, $w_\nu\in V_y\cap\Omega$ and $\overline{V_x}\cap\overline{V_y}=\emptyset$.
Hence, Claim~2 is true. \hfill $\btl$}
\smallskip

Combining this claim with the discussion that precedes it, the
result follows.
\end{proof}

We shall now undertake the proof of Theorem~\ref{T:local-global-weak-visibility}.
We will need the following lemmas to prove this theorem:

\begin{lemma}\label{L:kobayashi-local-strict-inequality}
Let $X$ be a complex manifold and $\Omega\subseteq X$ be a Kobayashi hyperbolic domain. Let $U$ be an
open subset of $X$ such that $U\cap\Omega \neq\emptyset$ with $U\cap\Omega$ connected. Then, for every 
$V\subset U$ with $V\cap\Omega\neq\emptyset$ and $K_\Omega(V\cap\Omega,\Omega\setminus U)>0$, there
exists a constant
$M>0$ such that for every $v\in T^{(1,0)}_{x}(V\cap\Omega)$ and every $x\in V\cap\Omega$,
\[
  k_\Omega(x;v)\leq k_{U\cap\Omega}(x;v)\leq\left(1+Me^{-K_\Omega(x,\Omega\setminus U)}\right)k_\Omega(x;v).
\]
\end{lemma}

In view of Result~\ref{R:royden-localization}, the argument for 
Lemma~\ref{L:kobayashi-local-strict-inequality} is as in the proof of \cite[Lemma~3.1]{sarkar:lkdavd23}.

\begin{lemma}\label{L:local-global-almost-geodesic}
Let $X$ be a complex manifold and $\Omega\varsubsetneq X$ be a Kobayashi hyperbolic domain. Let $U$, $V$,
$W$ be non-empty open 
subsets of $X$ such that $W\cap\Omega\neq\emptyset$, $W\Subset V\Subset U$ with 
$U\cap\Omega$ connected, and such that $\delta:=K_\Omega(V\cap\Omega,\Omega\setminus U)>0$. Let
$(U\cap\Omega,K_{U\cap\Omega})$ have the weak visibility property for every pair of distinct points in
$V\cap\bdy\Omega$.
Let $\kappa\geq 0$. Then, there exists a constant $\kappa_0>0$ such that for every $x,y\in
W\cap\Omega$ with $x\neq y$, and for
$\sigma: [0, T]\lrarw \Omega$ a $(1,\kappa)$-almost-geodesic with respect to $K_\Omega$ joining $x$
and $y$ such that the
image of $\sigma$ is contained in $W$,
\[
  l_{U\cap\Omega}\left(\sigma|_{[s_1, s_2]}\right) \leq K_{U\cap\Omega}(\sigma(s_1),
  \sigma(s_2))+\kappa_0
\]
for every $s_1\leq s_2\in [0, T]$.  
\end{lemma}

Essentially, following the same argument as in the proof of
\cite[Lemma~3.5]{sarkar:lkdavd23}\,---\,which is valid in the present setting in view of
Lemma~\ref{L:kobayashi-local-strict-inequality}\,---\,up to the estimate \cite[(3.6)]{sarkar:lkdavd23}
gives us the desired inequality.

\begin{proof}[The proof of Theorem~\ref{T:local-global-weak-visibility}]
First we will prove that $(a)$ implies $(b)$. \textbf{Fix} $\kappa\geq 0$. As in the proof of
Theorem~\ref{T:local-global-visibility}, considering the \textbf{same}
notations, we will prove Claim~1 occurring in its proof with $\lambda=1$. As before, we will assume
that the latter claim is false. This gives us
$\gamma_\nu$, $\nu = 1, 2, 3,\dots$, $(1,\kappa)$-almost-geodesics with respect to $K_\Omega$ joining
$x_\nu$ to $y_\nu$, where $x_\nu\in W_p\cap\Omega$, $y_\nu\in V_q\cap\Omega$, having \textbf{all} the
other properties listed in the proof of the above-mentioned Claim~1.
Let $r_\nu>0$ and ${y'_\nu}=\gamma_\nu(r_\nu)$ be such that 
${y'_\nu}\in (W_p^2\setminus\overline{W_p^1})\cap\Omega$,
and such that the image of $\wt{\gamma}_{\nu}:=\gamma_\nu{\mid}_{[0,r_\nu]}$ is contained in $W_p^2$. As
indicated above, $W_p^1$ and $W_p^2$ are the same as in the proof of the above-mentioned Claim~1,
and as argued in that proof using Lemma~\ref{L:3_nbhd}, we get $p'\in\overline{W_p}\cap\bdy\Omega$
such that $x_\nu\to p'$ and
\begin{align}\label{E:d_h-weak}
 \lim_{\nu\to\infty}\sup_{z\in\langle\wt\gamma_\nu\rangle}d_h(z,\bdy\Omega)=0.
\end{align}
By
Lemma~\ref{L:local-global-almost-geodesic}, there exists a constant $\kappa_0>0$ such that for every
$s_1,s_2\in[0,r_\nu]$ and every $\nu\geq 1$ we have
\begin{align}\label{E:l_bound}
 l_{U_p\cap\Omega}\left(\wt\gamma_\nu|_{[s_1, s_2]}\right)\leq K_{U_p\cap\Omega}
 (\wt\gamma_\nu(s_1), \wt\gamma_\nu(s_2))+\kappa_0. 
\end{align}
We claim the inequality \eqref{E:l_bound} implies that, for each $\nu$, there exists an
absolutely continuous path
$\sigma_\nu: [0,r'_\nu]\lrarw W_p^2$ that satisfies $\langle\sigma_{\nu}\rangle =
\langle\wt{\gamma}_{\nu} \rangle$ and is a $(1,\kappa_0)$-almost-geodesic with respect to $K_{U_p\cap\Omega}$
joining $x_\nu$ to $y_\nu'$. The construction of $\sigma_\nu$ involves the following two steps:
\begin{enumerate}[leftmargin=25pt]
  \item the construction of an auxiliary absolutely continuous path 
  $\wt{\gamma}_{\nu}^{{\rm aux}}$ with domain $[0, \rho_{\nu}]$\,---\,which just equals $\wt{\gamma}_{\nu}$
  if the set $\{t\in [0, r_{\nu}] : \wt{\gamma}'_{\nu}(t) = 0\}$ has empty
  interior\,---\,such that the set $\{t\in [0, \rho_{\nu}] : \left(\wt{\gamma}^{{\rm aux}}_{\nu}\right)'(t)
  = 0\}$ has empty interior and such that
  $\langle\wt{\gamma}_{\nu} \rangle = \langle \wt{\gamma}_{\nu}^{{\rm aux}}\rangle$.

  \item the reparametrization of $\wt{\gamma}_{\nu}^{{\rm aux}}$ by its Kobayashi--Royden arc-length in
  $U_p\cap\Omega$ (which makes sense because, by construction, the arc-length function
  \[
    G_{\nu}(t) := \int_0^t k_{U_p\cap\Omega}\left(\wt{\gamma}_{\nu}^{{\rm aux}}(s); 
    \left(\wt{\gamma}^{{\rm aux}}_{\nu}\right)'(s)\right)\,ds,
    \quad t\in [0, \rho_{\nu}],
  \]
  is invertible). 
\end{enumerate}
The reparametrized path in Step~$(2)$ is the $\sigma_{\nu}$ mentioned.
The construction in Step~$(1)$ is well-known. That the reparametrization described in Step~$(2)$ results in
an absolutely continuous path is non-obvious when $G_{\nu}^{-1}$ is
not absolutely continuous. We
refer the reader to \cite{bharalimasanta:nsphuswrtKm24} for the details of this construction. That
$\sigma_{\nu}$ is absolutely continuous is the outcome of
\cite[Proposition~4.2]{bharalimasanta:nsphuswrtKm24}. Then, the inequality \eqref{E:l_bound}, together with
\cite[Corollary~1.3$(a)$]{bharalimasanta:nsphuswrtKm24}, implies that $\sigma_{\nu}$ is a
$(1,\kappa_0)$-almost-geodesic with respect to $K_{U_p\cap\Omega}$.
As $W_p^2$ is relatively compact, there exists a sequence $(\nu_j)_{j\geq 1}\subset \N$ and a point
$q'\in\overline{W_p^2}\cap\bdy\Omega\subseteq V_p\cap\bdy\Omega$ such that $\gamma_{\nu_j}
(r_{\nu_j})=\sigma_{\nu_j}(r'_{\nu_j})\to q'$. By
construction, $p'\neq q'$. Since $V_p\cap\bdy\Omega$ satisfies the weak visibility condition with
respect to
$K_{U_p\cap\Omega}$, there exists a compact set $K\subset U_p\cap\Omega$ such that for every 
$j\geq 1$ we can find
$z_j\in\langle\sigma_{\nu_j}\rangle=\langle\wt\gamma_{\nu_j}\rangle$ with $z_j\in K$, 
which contradicts \eqref{E:d_h-weak}. Hence,
our claim is true; therefore, $\bdy\Omega$ is weakly visible. 
\smallskip

Now, to prove the converse, we only need to prove that $\Omega$ is a hyperbolically imbedded domain.
Note that, in
the argument for $(b)$ implies $(a)$ in the proof of Theorem~\ref{T:local-global-visibility},
we only need to work with $(1,\kappa)$-almost-geodesics. Therefore, that argument establishes the
implication that we
want. Hence, the result.
\end{proof}
\smallskip

\section{Compactifications and visibility domains}\label{S:compacti-visi}
With $X$ and $\Omega\varsubsetneq X$ as in the discussions above, recall what an ``admissible
compactification of $\overline{\Omega}$" means; see Definition~\ref{D:admi-comp}. If $\Omc$ is an
admissible compactification of $\overline\Omega$ and $\emb:\overline\Omega\hookrightarrow \Omc$ is the
embedding that is part of the definition of $\Omc$, then let us denote the set
$\Omc\setminus \emb(\Omega)$ by $\bdy_\infty\Omega$.
\smallskip

We now present the central result of this section, which will play a key supporting role in the
proofs presented in the subsequent sections. In what follows, given a set $S\subseteq\Omc$,
$\overline S^\infty$ will denote the closure of $S$ in $\Omc$. The sets
$\vbdy\Omega$ and $\wvbdy\Omega$ appearing in the result below are as defined in 
Definition~\ref{D:visible-point}.

\begin{theorem}\label{T:visible-point-visible-domain}
Let $X$ be a complex manifold and $\Omega\varsubsetneq X$ be a Kobayashi hyperbolic domain. Let
$\Omc$ be any admissible compactification of $\overline\Omega$. 
\begin{itemize}[leftmargin=25pt]
 \item[$(a)$] Suppose the set $\bdy\Omega\setminus\vbdy\Omega$ is totally disconnected. Then, $\Omega$
 is a visibility domain subordinate to $\Omc$.

  \item[$(b)$] Suppose the set $\bdy\Omega\setminus\wvbdy\Omega$ is totally disconnected. Then,
  $\Omega$ is a weak visibility domain subordinate to $\Omc$.
\end{itemize}
\end{theorem}

\begin{proof}
Fix a Hermitian metric $h$ on $X$ such that $(X, d_h)$, where $d_h$ is the distance induced
by $h$ on $X$, is Cauchy-complete. Fix
$p,q\in\bdy_\infty\Omega$ such that $p\neq q$. Let us first prove $(a)$. Therefore, fix $\lambda\geq 1$
and $\kappa\geq 0$. Let $V_p$, $V_q$ be two 
$\Omc$-open neighbourhoods such that $\overline{V_p}^\infty\cap\overline{V_q}^\infty=\emptyset$.
For any pair $x,y\in\Omega$ we define the set

\[
  AG_{\Omega}(x,y):=\{\gamma:[0,T]\lrarw\Omega:
  \text{$\gamma$ is a $(\lambda,\kappa)$-almost-geodesic joining $x$ and $y$}\}.
\]
Here, the parameters $\lambda$ and $\kappa$ are those that we fixed above.
\smallskip

We shall show that there exists a compact set $K\subset \Omega$ such that for every $x\in
\emb^{-1}(V_p\setminus \bdy_{\infty}\Omega)$, every
$y\in\emb^{-1}(V_q\setminus \bdy_{\infty}\Omega)$, and every $\gamma\in
AG_\Omega(x,y)$, we have $\langle\gamma\rangle\cap K\neq\emptyset$.
If possible, let this assertion be false. Let $(K_\nu)_{\nu\geq 1}$ be an exhaustion by compacts of
$\Omega$. Then, there exist sequences $(x_\nu)_{\nu\geq 1}\subset\emb^{-1}(V_p\setminus
\bdy_{\infty}\Omega)$,
$(y_\nu)_{\nu\geq 1}\subset\emb^{-1}(V_q\setminus \bdy_{\infty}\Omega)$,
and $(\lambda,\kappa)$-almost-geodesics $\gamma_\nu$ in 
$AG_\Omega(x_\nu,y_\nu)$, for each $\nu\geq 1$, such that $\langle\gamma_\nu\rangle\cap
K_\nu=\emptyset$. Let
$\gamma_\nu:[0,T_\nu]\lrarw\Omega$. Let us define
\[
  S_\infty := \{x \in \Omc : \exists (\nu_j)_{j\geq 1}\subset \N \ \text{and $t_j\in [0, T_{\nu_j}]$
  such that $\emb(\gamma_{\nu_j}(t_j))\lrarw x$ as $j\to \infty$}\}.
\]
Since $\Omc$ is sequentially compact, $S_\infty$ contains all subsequential limits (in $\Omc$) of
$(\emb(x_\nu))_{\nu\geq 1}$ and $(\emb(y_\nu))_{\nu\geq 1}$. As $(x_\nu)_{\nu\geq 1}\subset
\emb^{-1}(V_p\setminus \bdy_{\infty}\Omega)$, $(y_\nu)_{\nu\geq 1}\subset \emb^{-1}(V_q\setminus
\bdy_{\infty}\Omega)$, and as $\overline{V_p}^\infty\cap\overline{V_q}^\infty=\emptyset$, $S_\infty$
contains at least two elements. Also, since $\langle\gamma_\nu\rangle\cap K_\nu=\emptyset$ for all
$\nu\geq 1$, $S_\infty\subseteq \bdy_{\infty}\Omega$.
\medskip

\noindent{{\textbf{Claim.}}}
\emph{The set $S_\infty$ is connected.}
\smallskip

\noindent{\emph{Proof of Claim.}
If possible, let the claim be false. Then, there exist two disjoint $\Omc$-open sets $U,V$ such that
$S_\infty\subseteq U\cup V$, and $U\cap S_\infty\neq\emptyset$, $V\cap S_\infty\neq\emptyset$.
Therefore, since each
$S_{\nu}:=\gamma_\nu([0,T_\nu])$
is connected, for sufficiently large $\nu$, there exists a point (recall our mild abuse of notation described in
Section~\ref{SS:visi-compacti-WD}) $z'_\nu\in \emb(S_{\nu})\setminus(U\cup V)$. As $\Omc$ is sequentially compact,
there exists a subsequence ${(z'_{\nu_k})}_{k\geq 1}$ and a point $\wt{p}\in S_\infty$ such that
$z'_{\nu_k}\to \wt{p}$ with respect to the topology of $\Omc$. Then 
$\wt{p}\in S_\infty\setminus (U\cup V)$, which is impossible. Hence the claim. \hfill $\btl$}
\smallskip

Define $S:= \emb^{-1}(S_\infty)\subseteq \bdy\Omega$. If $S$ were at most finite, then $S_\infty$ would be
contained in 
$\bdy_{\infty}\Omega\setminus \emb(\bdy\Omega)$ \textbf{except} for at most finitely many points. This would
contradict the above claim since $|S_\infty|\geq 2$ and $\bdy_{\infty}\Omega\setminus \emb(\bdy\Omega)$, by definition, is totally disconnected. Thus $S$ is an infinite set; in particular $|S|\geq 2$.
\smallskip

We now wish to show that there exists a point $\xi\in S\cap\vbdy\Omega$. To this end, assume, if possible, that
$S\cap \vbdy\Omega = \emptyset$. Then, $S\subseteq\bdy\Omega\setminus\vbdy\Omega$: a totally disconnected set.
Hence, $S$ is totally disconnected. As $|S|\geq 2$, consider $x, y\in S$, $x\neq y$, and $r>0$ such that
$y\notin B_{d_h}(x,r)$ (see Remark~\ref{rem:open-d-h-ball-notation} for the latter notation). Write
\[
  K := \big(S\cap \overline{B_{d_h}(x, 2r)}\,\big)\setminus B_{d_h}(x,r)
\]
(to clarify, the closure is with respect to the topology on $X$). Since $S$ is totally disconnected, for each
$z\in K$, there exist disjoint subsets $A_z$ and $B_z$, open relative to $S$, such that $x\in A_z$, $z\in B_z$,
and $S = A_z\cup B_z$. By definition, $S$ is a closed subset of $\overline{\Omega}$. Now, recall that (see the
discussion in the proof of Proposition~\ref{P:almost-geodesic}) that the metric topology induced by $d_h$ equals
the topology on $X$. Thus, as $(X, d_h)$ is Cauchy-complete, it follows from the Hopf--Rinow theorem that $K$ is
compact. Thus, there exist $z_1,\dots, z_N\in K$ such that $K\subset \cup_{j=1}^N B_{z_j}$. In what follows in
this paragraph, we shall use more than once the facts\,---\,owing to the definition of $\Omc$ and the
observation that $S_\infty\cap \emb(E) = \emb(S)\cap \emb(E)$ for any set
$E\subseteq \overline{\Omega}$\,---\,that, under $\emb$, the image of any $S$-open subset is
$S_\infty$-open and the image of any compact subset of $\overline{\Omega}$ is closed in $\Omc$. Define
\[
  A := \emb\left(B_{d_h}(x,r)\cap \bigcap\nolimits_{1\leq j\leq N}A_{z_j}\right) \quad\text{and}
  \quad B = \big(S_{\infty}\setminus \emb(\overline{B_{d_h}(x,r)\cap \Omega})\big)
            \cup \emb\left(\bigcup\nolimits_{1\leq j\leq N}B_{z_j}\right).
\]
The sets $A$ and $B$ are disjoint subsets open relative to $S_\infty$ such that $S_\infty = A\cup B$,
$x\in A$, and $y\in B$. This contradicts the fact that $S_\infty$ is connected.
Therefore, there exists $\xi\in S\cap\vbdy\Omega$.
\smallskip

Since
$\overline{V_p}^\infty\cap\overline{V_q}^\infty=\emptyset$,
without loss of generality we can assume that
$\xi\notin\overline{V_q}^\infty$. Since $\xi\in S$, passing to a subsequence and relabelling, if needed,
we may assume that for each $\nu\geq 1$, there exists $r_\nu\in [0,T_\nu)$ such that
$\gamma_\nu(r_\nu)\to \xi$ as $\nu\to\infty$. Since $\xi\in\vbdy\Omega$ and
$\xi\notin\overline{V_q}^\infty$,
there exists a neighbourhood $U_\xi$ of $\xi$ in $X$ such that $U_\xi$ is relatively
compact, $\overline{U_\xi}\cap\overline{\emb^{-1}(V_q)}=\emptyset$, and such that every pair of distinct points
$\xi_1,\xi_2\in U_\xi\cap\bdy\Omega$ satisfies the visibility condition with respect to $K_\Omega$.
Also, fix neighbourhoods $W_{\xi}$, $V_{\xi}$ of $\xi$ such that $W_\xi\Subset V_\xi\Subset U_\xi$. 
We may assume that $\gamma_\nu(r_\nu)\in W_\xi\cap\Omega$ for all $\nu\geq 1$.
Let $s_\nu>r_\nu$ and ${y'_\nu}=\gamma_\nu(s_\nu)$ be such that 
${y'_\nu}\in (V_\xi\setminus\overline{W_\xi})\cap\Omega$, and such that the image of
$\wt{\gamma}_{\nu}:=\gamma_\nu{\mid}_{[r_\nu,s_\nu]}$ is contained in $V_\xi$. By Lemma~\ref{L:3_nbhd}
(and relabeling subsequences as usual) we have
\begin{align}\label{E:d_h-visibility}
 \lim_{\nu\to\infty}\sup_{z\in\langle\wt\gamma_\nu\rangle}d_h(z,\bdy\Omega)=0.
\end{align}
Therefore, there exists a sequence
$(\nu_j)_{j\geq 1}\subset \N$ and a point $\eta\in\overline{V_\xi}\cap\bdy\Omega\subseteq
U_\xi\cap\bdy\Omega$ such that
$y'_{\nu_j}\to \eta$.
By construction, $\xi\neq\eta$. Now, since $U_\xi\cap\bdy\Omega$
satisfies the visibility condition with respect to $K_\Omega$, there exists a compact set
$L\subset \Omega$ such that for
every $j\geq 1$ we can find $\wt z_{\nu_j}\in\langle\wt\gamma_{\nu_j}\rangle$ with $\wt z_{\nu_j}\in L$, 
which contradicts \eqref{E:d_h-visibility}.
This establishes $(a)$.
\smallskip

It is evident that, owing to our definitions, the entire argument above would be valid if we fix
$\lambda = 1$ and replace $\vbdy\Omega$ by $\wvbdy\Omega$. Thus $(b)$ follows.
\end{proof}

The above result gives us the following corollaries.

\begin{corollary}\label{C:visible-visibility-domain}
Let $X$ be a complex manifold and $\Omega\varsubsetneq X$ be a Kobayashi hyperbolic domain. Then, the following are
equivalent:
\begin{itemize}[leftmargin=25pt]
 \item[$(a)$]$\bdy\Omega$ is visible.
 \item[$(b)$]$\Omega$ is a visibility domain subordinate to any admissible compactification of
 $\overline\Omega$.
 \item[$(c)$]$\Omega$ is a visibility domain subordinate to some admissible compactification of
 $\overline\Omega$.
\end{itemize}
\end{corollary}

\begin{proof}
Let $\bdy\Omega$ be visible. Then, every point in $\bdy\Omega$ is a visible point of $\Omega$.
Therefore, $(a)$ implies $(b)$ follows from Theorem~\ref{T:visible-point-visible-domain}-$(a)$. That $(b)$
implies $(c)$ is trivial. Finally, $(c)$ implies $(a)$ follows by Definition~\ref{D:visible}.  
\end{proof}

\begin{corollary}\label{C:visible-weak-visibility-domain}
Let $X$ be a complex manifold and $\Omega\varsubsetneq X$ be a Kobayashi hyperbolic domain. Then, the following are
equivalent:
\begin{itemize}[leftmargin=25pt]
 \item[$(a)$]$\bdy\Omega$ is weakly visible.
 \item[$(b)$]$\Omega$ is a weak visibility domain subordinate to any admissible compactification of
 $\overline\Omega$.
 \item[$(c)$]$\Omega$ is a weak visibility domain subordinate to some admissible compactification of
 $\overline\Omega$.
\end{itemize}
\end{corollary}

The proof of the above is analogous to that of Corollary~\ref{C:visible-visibility-domain}.
\smallskip

Note that, in view of the above corollaries, Theorem~\ref{T:equivalence} follows immediately.
\smallskip

Finally, we can provide a proof of Proposition~\ref{P:visible-points-empty}. 

\begin{proof}[The proof of Proposition~\ref{P:visible-points-empty}]
If possible let $\bdy\Omega\setminus\vbdy\Omega$ be a non-empty, totally disconnected set. Then, by
Theorem~\ref{T:visible-point-visible-domain} and  Corollary~\ref{C:visible-visibility-domain}, we have
$\bdy\Omega$ is visible. Therefore, by Definition~\ref{D:visible-point}, it follows that
$\bdy\Omega=\vbdy\Omega$, which contradicts the fact that $\bdy\Omega\setminus\vbdy\Omega$ is 
non-empty. Similarly, we can prove that $\bdy\Omega\setminus\wvbdy\Omega$ cannot be a non-empty,
totally disconnected set.     
\end{proof}
\smallskip

\section{Sufficient conditions for visibility}\label{S:sufficint-condition}
In this section, we shall present several sufficient conditions for a domain $\Omega\varsubsetneq X$, 
$\Omega$ and $X$ as in the previous sections, such that $\bdy\Omega$ is visible. We will first begin
with a sufficient condition that is inspired by a construction in \cite{bharalizimmer:gdwnv23}\,---\,i.e.,
involving the notion of a ``local Goldilocks point'' alluded to in Section~\ref{SS:sufficient}.
Thereafter, we will provide proofs of the results stated in Section~\ref{SS:sufficient}.%
\smallskip

We first need a few definitions and observations.

\begin{definition}\label{D:M(r)}
Let $X$ be a complex manifold and $\Omega\varsubsetneq X$ be a Kobayashi hyperbolic domain. Fix
$p\in\bdy\Omega$. Let $(U,\varphi)$ be a holomorphic chart of $X$ centered at $p$. We define
\begin{multline*}
 \qquad\qquad M(r;\varphi,U):=\sup\left\{\frac{1}{k_\Omega(x;v)}: \text{$v\in T_x^{(1,0)}\Omega$ such that
 $\|\varphi'(x)v\|=1$} \right. \\
  \text{and $x\in U\cap\Omega$ such that $d_{{\rm
  Euc}}(\varphi(x),\varphi(U\cap\bdy\Omega))\leq r$}
  \bigg\}.
\end{multline*}
\end{definition}

The function $M(\bcdot;\varphi,U)$ depends on the choice of charts, which could potentially pose
difficulties in any
analysis purportedly about $\Omega$ but which relies on estimates for $M(\bcdot;\varphi,U)$ for some
specific $(U,\varphi)$. The following lemma is a step towards addressing
this issue. In this paper, we abbreviate $d_{Euc}(z,S)$ as $\delta_S(z)$, for any 
$z\in\Cn$ and $S\subseteq\Cn$. The notation $\|\bcdot\|$ denotes the Euclidean norm on $\Cn$.

\begin{lemma}\label{L:local-goldilocks-chart}
Let $X$, $\Omega$, and $p\in\bdy\Omega$ be as in Definition~\ref{D:M(r)}. Let $(U,\Phi)$ and
$(V,\Psi)$ be two holomorphic
charts of $X$ centered at $p$. Then, there exists a neighbourhood $\Delta$ of $p$ contained in $U\cap V$
and constants $R,\beta\geq 1$, which depend only on $\Phi$, $\Psi$, and $\Delta$ such that
\[
  M(r;\Phi,W)\leq \beta M(Rr;\Psi, W) \quad\text{and} \quad
  \delta_{\Psi(W\cap\bdy\Omega)}(\Psi(x))\leq R\delta_{\Phi(W\cap\bdy\Omega)}(\Phi(x))
\]
for every neighbourhood $W$ of $p$ such that $W\subseteq \Delta$, every $r>0$, and for every $x\in W\cap\Omega$. 
\end{lemma}
\begin{proof}
Let $\Delta$ be a neighbourhood of $p$ such that $\overline{\Delta}\subset U\cap V$ and such
that $\overline{\Phi(\Delta)}$ is a closed polydisk with centre $0\in \C^n$. Fix a neighbourhood $W$ of
$p$ such that $W\subseteq \Delta$. Let $x\in W\cap\Omega$ and
$r>0$  be such that $\delta_{\Phi(W\cap\bdy\Omega)}(\Phi(x))\leq r$. Then, there exists 
$\xi\in \overline{W}\cap\bdy\Omega$ such that
$\|\Phi(x)-\Phi(\xi)\| = \delta_{\Phi(W\cap\bdy\Omega)}(\Phi(x))$.
Since $\Psi\circ(\Phi{\mid}_{\overline{\Delta}})^{-1}$ is the restriction of a
biholomorphic map on $\Phi(U\cap V)$ and $\overline{\Delta}$ is compact, there exists a constant 
$R\geq 1$ such 
that the operator norm $\|(\Psi\circ(\Phi{\mid}_{\overline{\Delta}})^{-1})'(\zeta)\|_{{\rm OP}}\leq R$
for all $\zeta\in \Phi(\overline{\Delta})$. In what follows, we shall abbreviate
$\Psi\circ(\Phi{\mid}_{\overline{\Delta}})^{-1}$ to
$\Psi\circ\Phi^{-1}$ for simplicity of notation. Since the line
segment joining $\Phi(x)$ and $\Phi(\xi)$ lies in $\overline{\Phi(\Delta)}
= \Phi(\overline{\Delta})$, we have
\begin{align}
 \|\Psi(x)-\Psi(\xi)\|&=\|(\Psi\circ\Phi^{-1})(\Phi(x))-(\Psi\circ\Phi^{-1})
 (\Phi(\xi))\|\notag\\
 &\leq\sup_{\zeta\in\Phi(\overline{\Delta})}
 \| (\Psi\circ\Phi^{-1})'(\zeta)\|_{{\rm OP}}\|(\Phi(x)-
 \Phi(\xi))\|\notag\\
 &\leq R\delta_{\Phi(W\cap\bdy\Omega)}(\Phi(x))\label{E:local-goldilocks-delta}.
\end{align}
It follows that $\delta_{\Psi(W\cap\bdy\Omega)}(\Psi(x))\leq
R\delta_{\Phi(W\cap\bdy\Omega)}(\Phi(x))$.
Now, there exists a constant $\beta\geq 1$ such that for any $z\in \overline{\Delta}\cap\Omega$ and 
every $v\in T_z^{(1,0)}\Omega$ that satisfies $\|\Phi'(z)v\|=1$, we have $1/\beta\leq\|\Psi'(z)v\|
\leq\beta$. Since, 
\begin{itemize}
  \item for each $v\in T_z^{(1,0)}\Omega$ that satisfies $\|\Phi'(z)v\|=1$, if
  $\wt{v} := v/\|\Psi'(z)v\|$, then $\|\Psi'(z)\wt{v}\| = 1$,
  \item the argument culminating in \eqref{E:local-goldilocks-delta} implies that
  $\{x\in W\cap\Omega : \delta_{\Phi(W\cap\bdy\Omega)}
  (\Phi(x))\leq r\}$\linebreak $\subseteq \{x\in W\cap\Omega : \delta_{\Psi(W\cap\bdy\Omega)}
  (\Psi(x))\leq Rr\}$,
\end{itemize}
the identity $k_{\Omega}(z; c\,\bcdot) = |c|k_{\Omega}(z; \bcdot)$ on $T_z^{(1,0)}\Omega$ (for any 
$z\in \Omega$ and any $c\in \C$) implies that
\[
  M(r;\Phi,W)\leq \beta M(Rr;\Psi, W).
\]
Note that the constants $R$ and $\beta$ are determined by $\Delta$ but (note the estimate
\eqref{E:local-goldilocks-delta}) apply to every $W$. Hence the proof.
\end{proof}

\begin{definition}\label{D:local-goldilocks-point}
Let $X$ be a complex manifold and $\Omega\varsubsetneq X$ be a Kobayashi hyperbolic domain. A boundary
point $p\in\bdy\Omega$ is called a \emph{local Goldilocks point} if there exists a holomorphic chart
$(\mathcal{U},\varphi)$ centered at $p$ and a neighbourhood $U_0$ of $p$ contained in  
$\mathcal{U}$ such that for any neighbourhood $U$ of $p$ such that $U\subseteq U_0$:
\begin{itemize}[leftmargin=25pt]
 \item[$(a)$] for some $\varepsilon>0$ we have
 \[
   \int_{0}^{\varepsilon}\frac{1}{r}M(r;\varphi,U)\;dr<\infty,
\]
 \item[$(b)$] for each $z_0\in\Omega$ there exist constants $C,\alpha>0$ (which depend on
 $z_0,U,\varphi$) such that for
 all $x\in U\cap\Omega$,
 \[
   K_\Omega(z_0,x)\leq C+\alpha\log\frac{1}{\delta_{\varphi(U\cap\bdy\Omega)}(\varphi(x))}.
 \]
\end{itemize}
Let $\bdy_{{\rm lg}}\Omega\subseteq\bdy\Omega$ denote the set of all local Goldilocks
points.
\end{definition}

\begin{remark}\label{rem:local-goldilocks}
Note that the function $M(\bcdot;\varphi,U)$ is a non-negative, monotone increasing function, owing to
which the integral in $(a)$ makes sense. Also note that the status of whether a point
$p\in\bdy\Omega$ is a local Goldilocks point is independent of the choice of chart featured in
Definition~\ref{D:local-goldilocks-point}. To see this, let $(\mathcal{U}, \varphi)$ be as in the latter
definition and consider another holomorphic chart $(\mathcal{W}, \psi)$ centered at $p$. Let:
\begin{itemize}
  \item $\Delta_1$, $R_1$, and $\beta_1$ be the objects provided by Lemma~\ref{L:local-goldilocks-chart}
  if we take $(U,\Phi)  = (\mathcal{W},\psi)$, and $(V,\Psi) = (U_0,\varphi)$.
  \item $\Delta_2$, $R_2$ and $\beta_2$ be the objects provided by Lemma~\ref{L:local-goldilocks-chart}
  if we take $(U,\Phi)  = (U_0,\varphi)$ and $(V,\Psi) = (\mathcal{W},\psi)$.
\end{itemize}
Here $U_0\subseteq \mathcal{U}$ is as in Definition~\ref{D:local-goldilocks-point}. 
Define $W_0 := \Delta_1\cap \Delta_2$, which is a neighbourhood of $p$ contained in $\mathcal{W}$. Fix a
neighbourhood $W$ of $p$ such that $W\subseteq W_0$. Since $W\subseteq U_0$, by definition,
 \[
   \int_{0}^{\varepsilon}\frac{1}{r}M(r;\varphi,W)\;dr<\infty
\]
for some $\varepsilon > 0$. Then, as $W\subseteq \Delta_1$, by Lemma~\ref{L:local-goldilocks-chart}:
\[
  \infty > \beta_1\int_{0}^{\varepsilon}\frac{1}{r}M(r;\varphi,W)\;dr
  \geq \int_{0}^{\varepsilon}\frac{1}{r}M(r/R_1;\psi,W)\;dr \\
  = \int_{0}^{\varepsilon/R_1}\frac{1}{r}M(r;\psi,W)\;dr.
\]
Fix $z_0\in \Omega$. By definition (since $W\subseteq U_0$), for all $x\in W\cap\Omega$,
\[
  K_\Omega(z_0,x)\leq C+\alpha\log\frac{1}{\delta_{\varphi(W\cap\bdy\Omega)}(\varphi(x))}
  \leq (C+\alpha\log(R_2)) + \alpha\log\frac{1}{\delta_{\psi(W\cap\bdy\Omega)}(\psi(x))}.
\]
The last inequality is due to Lemma~\ref{L:local-goldilocks-chart}, which applies as
$W\subseteq \Delta_2$. So, with $W_0$ taking the role of $U_0$ in
Definition~\ref{D:local-goldilocks-point}, we see that the conditions~$(a)$
and~$(b)$ in the latter definition hold true with $\psi$ and $W$ replacing $\varphi$ and $U$, respectively.
In short: the status of whether a point $p\in\bdy\Omega$ is a local Goldilocks point is independent
of the choice of holomorphic charts centered at $p$.
\end{remark}

If $(\mathcal{U},\varphi)$ is a chart centered at $p\in \bdy\Omega$ and satisfies
the conditions in Definition~\ref{D:local-goldilocks-point}, then we just saw that in arguing that
any other holomorphic chart centered at $p$ satisfies conditions analogous to those on $(\mathcal{U},
\varphi)$, we require that the conditions~$(a)$ and~$(b)$ in the latter definition hold on neighbourhoods $U$
of $p$ that are possibly \emph{smaller} than $\mathcal{U}$. The latter are the neighbourhoods we shall work
with in the following proofs. As coordinate charts, these, in reference to
Definition~\ref{D:local-goldilocks-point}, should be $(U, \varphi{\mid}_{U})$ but, for simplicity of notation,
these will be denoted by $(U,\varphi)$.

\begin{lemma}\label{L:local-goldilocks}
Let $X$ be a complex manifold and $\Omega\varsubsetneq X$ be a Kobayashi hyperbolic
domain. Let $p\in\bdy_{{\rm lg}}\Omega$. Let $(U,\varphi)$ be a holomorphic chart centered at $p$ that
satisfies the conditions~$(a)$ and~$(b)$ in
Definition~\ref{D:local-goldilocks-point} and such that $U$ is relatively compact. Fix $\lambda\geq 1$
and $\kappa\geq 0$. Suppose for each $\nu\geq 1$, $\gamma_\nu:[0,r_\nu]\lrarw\Omega$ is a 
$(\lambda,\kappa)$-almost-geodesic such that $\gamma_\nu([0,r_\nu])\subset U$, and such that
$\gamma_\nu(0)\to p'\in U\cap\bdy\Omega$, $\gamma_\nu(r_\nu)\to q'\in U\cap\bdy\Omega$ as
$\nu\to\infty$
with $p'\neq q'$. Let $(\varphi\circ\gamma_\nu)_{\nu\geq 1}$ converge locally uniformly to a curve 
$\sigma:[0,M)\lrarw\overline{\varphi(U)}$, for some $M\in(0,\infty]$. Then, $\sigma$ is non-constant.
\end{lemma}

\begin{proof}
Define $\sigma_\nu:=\varphi\circ\gamma_\nu$ for every $\nu\geq 1$ and $Z:=\varphi(U\cap\bdy\Omega)$.
Fix $z_0\in U\cap\Omega$.
From part $(b)$ of Definition~\ref{D:local-goldilocks-point} there exist constants $C,\alpha>0$ (which
depend on $z_0,U,\varphi$) such that for
all $x\in U\cap\Omega$
 \[
   K_\Omega(z_0,x)\leq C+\alpha\log\frac{1}{\delta_{\varphi(U\cap\bdy\Omega)}(\varphi(x))}.
\]
Therefore, since each $\gamma_\nu$ is a $(\lambda,\kappa)$-almost-geodesic, for every
$t\in[0,r_\nu]$, every $\nu\geq 1$, we have
\begin{align*}
 \frac{1}{\lambda}|t|-\kappa\leq K_\Omega(\gamma_\nu(0),\gamma_\nu(t))&\leq
 K_\Omega(\gamma_\nu(0),z_0)+K_\Omega(z_0,\gamma_\nu(t))\\&\leq 2C+\alpha\log\frac{1}
 {\delta_Z(\sigma_\nu(0))\delta_Z(\sigma_\nu(t))}.  
\end{align*}
 Now, at this stage we can proceed exactly as in the proof of \cite[Lemma~6.1]{{bharalizimmer:gdwnv23}}
 and hence, the result follows.
\end{proof}

We are now in a position to give a sufficient condition for visibility. This result is inspired by
Theorem~1.4 in \cite{bharalizimmer:gdwnv23} and (in view of Corollary~\ref{C:visible-visibility-domain})
is the exact analogue of the latter theorem.

\begin{theorem}\label{T:local-goldilocks}
Let $X$ be a complex manifold and $\Omega\varsubsetneq X$ be a Kobayashi hyperbolic domain. Suppose the
set $\bdy\Omega\setminus\bdy_{{\rm lg}}\Omega$ is totally disconnected. Then, $\bdy\Omega$ is 
visible.  
\end{theorem}

\begin{proof}
Fix a Hermitian metric $h$ on $X$ and let $d_h$ be the distance induced by $h$ on $X$.
\medskip

\noindent{{\textbf{Claim.}}}
\emph{Every local Goldilocks point is a visible point.}
\smallskip

\noindent{\emph{Proof of Claim.}
Let $p\in\lbdy\Omega$. Therefore,
there exists a holomorphic chart $(U,\varphi)$ centered at $p$ that satisfies the
conditions~$(a)$ and~$(b)$
in Definition~\ref{D:local-goldilocks-point}. Without loss of generality, we can assume that
$\varphi$ provides holomorphic coordinates on a larger neighbourhood in which $U$ is
relatively compact. Write $\wt U:=\varphi(U)$, which is bounded. Fix $\xi,\eta\in U\cap\bdy\Omega$ such that
$\xi\neq\eta$. We also fix $\lambda\geq 1$ and $\kappa\geq 0$. Let $V_\xi$, $V_\eta$ be two
neighbourhoods of $\xi$ and $\eta$ respectively such that
$\overline{V_\xi}\cap\overline{V_\eta}=\emptyset$ and $\overline{V_\xi}\subset U$. Let $W_\xi$ be a
neighbourhood of $\xi$ such that
$\overline{W_\xi}\subset V_\xi$. For any pair $x,y\in\Omega$,
let $AG_{\Omega}(x,y)$ have the same meaning as in Section~\ref{S:proof-local-golbal-visibility}.}
\smallskip

We will
show that there exists a compact set $K\subset \Omega$ such that for 
every $x\in
W_\xi\cap\Omega$, every
$y\in V_\eta\cap\Omega$, and every $\gamma\in AG_\Omega(x,y)$, we have $\langle\gamma\rangle\cap K\neq\emptyset$.
If possible, let this assertion be false. Let $(K_\nu)_{\nu\geq 1}$ be an exhaustion by compacts of
$\Omega$. Then, there exist sequences $(x_\nu)_{\nu\geq 1}\subset W_\xi\cap\Omega$,
$(y_\nu)_{\nu\geq 1}\subset V_\eta\cap\Omega$, and $(\lambda,\kappa)$-almost-geodesics $\gamma_\nu$ in 
$AG_\Omega(x_\nu,y_\nu)$, for each $\nu\geq 1$, such that $\langle\gamma_\nu\rangle\cap
K_\nu=\emptyset$.
Fix neighbourhoods $W_{\xi}^1$, $W_{\xi}^2$ of $\xi$ such that $W_\xi\Subset W_\xi^1\Subset
W_\xi^2\Subset V_\xi.$
Let $r_\nu>0$ and ${y'_\nu}=\gamma_\nu(r_\nu)$ be such that ${y'_\nu}\in
(W_\xi^2\setminus\overline{W_\xi^1})\cap\Omega$, and
such that the image of $\wt{\gamma}_{\nu}:=\gamma_\nu{\mid}_{[0,r_\nu]}$ is contained in $W_\xi^2$.
By Lemma~\ref{L:3_nbhd}, and as we can relabel the subsequences given by it, there exist a point
$p'\in \overline{W_\xi}\cap\bdy\Omega$
such that $x_{\nu}\to p'$ and
\begin{align}\label{E:d_h-local-goldilocks}
 \lim_{\nu\to\infty}\sup_{z\in\langle\wt\gamma_\nu\rangle}d_h(z,\bdy\Omega)=0.
\end{align}
Therefore, there exists a sequence
$(\nu_j)_{j\geq 1}\subset \N$ and a point $q'\in\overline{W_\xi^2}\cap\bdy\Omega\subseteq
V_\xi\cap\bdy\Omega$ such that
$\gamma_{\nu_j}(r_{\nu_j})\to q'$. By construction, $p'\neq q'$. Define
$\sigma_j:=\varphi\circ\wt\gamma_{\nu_j}$ for every $j\geq 1$ and let $Z:=\varphi(U\cap\bdy\Omega)$.
\smallskip

Since $p$ is a local Goldilocks point, from part $(a)$ of Definition~\ref{D:local-goldilocks-point} it
follows that $M(r;\varphi,U)\to 0$ as $r\to 0^+$. Let $\delta> 0$ be such that whenever
$r\in(0,\delta)$, we have
$0\leq M(r;\varphi,U)<1$. Therefore, as $k_{\Omega}(z; c\,\bcdot) = |c|k_{\Omega}(z; \bcdot)$ 
on $T_z^{(1,0)}\Omega$ for 
$z\in \Omega$ and $c\in \C$, for every $v\in T^{(1,0)}_z\Omega$, every $z\in U\cap\Omega$ that
satisfies $\delta_Z(\varphi(z))<\delta$, we have $k_\Omega(z;v)\geq\|\varphi'(z)v\|$. Now, let
$V:=U\cap\{z\in\Omega:\delta_Z(\varphi(z))\geq\delta\}$. Since $U$ is relatively compact,
$\overline{V}\subset \Omega$ and $\overline{V}$ is compact. Therefore, by 
Result~\ref{R:kob-hyp-cond} and our choice of $U$, there exists a constant $L_0>0$ such that
for every $v\in
T^{(1,0)}_{z}\Omega$, every $z\in V$, we have $k_\Omega(z;v)\geq L_0\|\varphi'(z)v\|$. Let 
$L:=\min\{L_0,1\}$. Since $\wt\gamma_\nu$ is a $(\lambda,\kappa)$-almost-geodesic, from above it
follows that 
\begin{align*}
 \|\sigma_j'(t)\|=\|\varphi'(\wt\gamma_{\nu_j}(t)){\wt\gamma_{\nu_j}}'(t)\|\leq\frac{1}
 {L}k_\Omega(\wt\gamma_{\nu_j}(t);
 \wt\gamma_{\nu_j}'(t))\leq\lambda/L,
\end{align*}
for almost every $t\in[0,r_{\nu_j}]$, every $j\geq 1$.
Therefore, for every $s,r\in[0,r_{\nu_j}]$, every $j\geq 1$ we have
\begin{align*}
 \sigma_j(r) &=\sigma_j(s) + \int_{s}^{r}\sigma_j'(t)\;dt , \ \text{which}\\ 
 \text{implies} \ \|\sigma_j(r)-\sigma_j(s)\|&\leq\frac{\lambda}{L}|r-s|.
\end{align*}
Therefore, $\sigma_j$ is $\lambda/L$-Lipschitz for every $j\geq 1$. Therefore, since
$\langle\sigma_j\rangle\subset \wt U$, by Arzel{\'a}--Ascoli Theorem, passing to a subsequence and
relabelling, if needed, we can assume that $\sigma_j$ converges locally uniformly to a curve 
$\sigma:[0,M)\lrarw\overline{\wt U}$ (where $M\in(0,\infty]$). We will show that for almost every
$t\in[0,r_{\nu_j}]$, every $j\geq 1$
\begin{align}\label{E:sigma-constant}
 \|\sigma_j'(t)\|\leq\lambda M(\delta_Z(\sigma_j(t));\varphi,U).  
\end{align}
Now, \eqref{E:sigma-constant} is true when $\|\sigma_j'(t)\|=0$. Let $\|\sigma_j'(t)\|\neq 0$. Since
$\wt\gamma_{\nu_j}$ is a $(\lambda,\kappa)$-almost-geodesic, for almost every $t\in[0,r_{\nu_j}]$, every
$j\geq 1$, we have
\begin{align*}
 k_\Omega(\wt\gamma_{\nu_j}(t);\wt\gamma_{\nu_j}'(t))&\leq\lambda, \ \text{which} \notag\\
 \text{implies} \ \|\sigma_j'(t)\|&\leq\frac{\lambda}{k_\Omega\bigg(\wt\gamma_{\nu_j}
 (t);\frac{\wt\gamma_{\nu_j}'(t)}{\|\sigma_j'(t)\|}\bigg)}\leq\lambda
 M(\delta_Z(\sigma_j(t));\varphi,U).
\end{align*}
Now, by \eqref{E:d_h-local-goldilocks} we may assume, passing to a subsequence and
relabelling, if needed, that
$\sup_{\zeta\in\langle\sigma_j\rangle}\delta_Z(\zeta) \searrow 0$ as $j\to 0$. From this, it follows
that $\lim_{j\to\infty}\sup_{\zeta\in\langle\sigma_j\rangle}
M(\delta_Z(\zeta);\varphi,U)=0$ since $M(\bcdot;\varphi,U)$ is monotone decreasing.
Hence, for every $s,r\in[0,M)$ we have
\begin{align*}
 \|\sigma(s)-\sigma(r)\|=\lim_{j\to\infty}\|\sigma_j(s)-
 \sigma_j(r)\|&\leq\limsup_{j\to\infty}\int_{s}^{r}\|\sigma_j'(t)\|\;dt\\ 
 &\leq\lambda\limsup_{j\to\infty}\int_{s}^{r}M(\delta_Z(\sigma_j(t));\varphi,U)\;dt=0.  
\end{align*}
Therefore, $\sigma$ is constant, which contradicts Lemma~\ref{L:local-goldilocks}. Hence, the claim.
\hfill $\btl$
\smallskip

By the above claim, since $\bdy\Omega\setminus\lbdy\Omega$ is totally disconnected, we
conclude that $\bdy\Omega\setminus\vbdy\Omega$ is totally disconnected. Therefore, by 
Theorem~\ref{T:visible-point-visible-domain} and Corollary~\ref{C:visible-visibility-domain}, 
we have the result.
\end{proof}

At this stage, we can give proofs of some more sufficient conditions for visibility. These are the results
stated in Section~\ref{SS:sufficient}. We begin with proving 
Theorem~\ref{T:visible-around-p-totally-disconnected}. 

\begin{proof}[The proof of Theorem~\ref{T:visible-around-p-totally-disconnected}]
Fix a Hermitian metric $h$ on $X$ and let $d_h$ be the distance induced by $h$ on $X$. Since
$\Omega\varsubsetneq X$ is relatively compact, by \cite[Result~3.1, Proposition~2.4]{masanta:kcdcm23},
we can conclude that $\Omega$ is a hyperbolically imbedded domain.
\medskip

\noindent{{\textbf{Claim.}}}
\emph{Each point $p$ in $\bdy_V\Omega$ is a visible point of $\bdy\Omega$.}
\smallskip

\noindent{\emph{Proof of Claim.}
Let $p\in\bdy_V\Omega$. Therefore, there exists a holomorphic chart $(U,\varphi)$ centered at $p$ such
that $\varphi(U)$ is bounded, $\varphi(U\cap\Omega)$ is a Kobayashi hyperbolic domain, and such that
$\bdy(\varphi(U\cap\Omega))$ is visible. Let
$G:=\varphi(U\cap\Omega)$.
Fix $\xi,\eta\in U\cap\bdy\Omega$ such that $\xi\neq\eta$. Therefore,
$\varphi(\xi),\varphi(\eta)\in\bdy G$ such that $\xi'=\varphi(\xi)\neq\varphi(\eta)=\eta'$. We also fix
$\lambda\geq 1$ and $\kappa\geq 0$. Since $\bdy G$ is visible, the pair $(\xi',\eta')$ satisfies the
visibility condition with respect to $K_G$. Therefore, since $\varphi$ is a biholomorphism on $U$ and
$\xi,\eta\in\bdy(U\cap\Omega)$, the pair ($\xi,\eta$) satisfies the visibility condition with
respect to $K_{U\cap\Omega}$. Therefore, $\bdy\Omega$ is locally visible at $p$.
\smallskip

Since $\Omega$ is a hyperbolically imbedded domain in $X$, there exists a neighbourhood
$V_\xi\subset U$ of $\xi$ in $X$ such that
$K_\Omega(V_\xi\cap \Omega, \Omega\setminus U):=\delta_\xi>0$ and that $\eta\notin\overline{V_\xi}$.
Without loss of generality, we may assume that $V_\xi$ is relatively compact.
Since $\eta\notin\overline{V_\xi}$, there exists a
neighbourhood $V_\eta$ of $\eta$ in $X$ such that $V_\eta$ is relatively compact, and such that
$\overline{V_\eta}\cap\overline{V_\xi}=\emptyset$.  Let $W_\xi$ be a neighbourhood of $\xi$ in $X$ such
that $\overline{W_\xi}\subset V_\xi$.
For any pair $x,y\in\Omega$ let $AG_\Omega(x,y)$ 
have the same meaning as in
Section~\ref{S:proof-local-golbal-visibility}. It suffices to show that there exists a compact set
$K\subset \Omega$ such that for 
every $x\in W_\xi\cap\Omega$, every
$y\in V_\eta\cap\Omega$, and every $\gamma\in AG_\Omega(x,y)$, we have $\langle\gamma\rangle\cap K\neq\emptyset$. At this stage, the existence
of such a $K$ follows by the same argument as in the proof of Claim~1 of the proof of
Theorem~\ref{T:local-global-visibility}. Therefore, $p$ is a visible point of $\bdy\Omega$ and hence,
the claim. \hfill $\btl$  
}
\smallskip

By the above claim, since $\bdy\Omega\setminus\bdy_V\Omega$ is totally disconnected, we
conclude that $\bdy\Omega\setminus\vbdy\Omega$ is totally disconnected. Therefore, by 
Theorem~\ref{T:visible-point-visible-domain} and Corollary~\ref{C:visible-visibility-domain}, 
part~($i$) follows.

\smallskip

It is evident that owing to our definitions, the entire argument above would be valid if we fix
$\lambda = 1$ and replace the reference to Theorem~\ref{T:local-global-visibility} with
Theorem~\ref{T:local-global-weak-visibility}, and substitute $\bdy\Omega\setminus\bdy_V\Omega$ with
$\bdy\Omega\setminus\bdy_{WV}\Omega$. Thus, $(ii)$ follows.
\end{proof}

Next, we will provide a proof of Theorem~\ref{T:domain-totally-discoonected-visible}. This result is
inspired by Banik's result \cite[Theorem~3.1]{banik:vdtanp23}. However, there are two differences between
the hypotheses of our result and \cite[Theorem~3.1]{banik:vdtanp23}. Firstly, our setting is domains in
arbitrary complex manifolds that are not necessarily relatively compact. Secondly, instead of removing a
compact set contained in the ``old'' domain, we remove a closed (not necessarily compact) totally
disconnected set. Consequently, there are significant differences between the proofs of
Theorem~\ref{T:domain-totally-discoonected-visible} and \cite[Theorem~3.1]{banik:vdtanp23}.

\begin{proof}[The proof of Theorem~\ref{T:domain-totally-discoonected-visible}]
Fix a Hermitian metric $h$ on $X$ and let $d_h$ be the distance induced by $h$ on $X$. Since $\Omega$
is a Kobayashi hyperbolic domain and $\Omega\setminus S$ is connected, it follows that
$\Omega\setminus S$ is a Kobayashi hyperbolic domain. Since $S$ is closed and $\mathcal{H}^{2n-2}(S)=0$,
by \cite[Theorem~3.2.22]{kobayashi:hcs98} we have
\begin{align}\label{E:kobayshi-distance-set}
 K_{\Omega\setminus S}=K_\Omega{\mid}_{(\Omega\setminus S)\times(\Omega\setminus S)}.
\end{align}
\medskip

\noindent{{\textbf{Claim.}}}
\emph{Each point $p$ in $\bdy\Omega\setminus S$ is a visible point of $\bdy(\Omega\setminus S)$.}
\smallskip

\noindent{\emph{Proof of Claim.}
Let $p\in\bdy\Omega\setminus S\subseteq\bdy(\Omega\setminus S)$. Since $S$ is closed, there exists a
neighbourhood $U_p$ of $p$ in $X$
such that $U_p$ is relatively compact and $\overline{U_p}\cap S=\emptyset$. Therefore,
$U_p\cap\bdy(\Omega\setminus S)=U_p\cap\bdy\Omega$. Fix $\lambda\geq 1$ and $\kappa\geq 0$. We also
fix $\xi,\eta\in U_p\cap\bdy(\Omega\setminus S)$ such that $\xi\neq\eta$. There exist two
neighbourhoods $V_\xi$,$V_\eta$ of $\xi$, $\eta$ respectively in $X$ such that
$\overline{V_\xi}\cap\overline{V_\eta}=\emptyset$ and such that $\overline{V_\xi}$,
$\overline{V_\eta}\subset U_p$. Let $W_\xi$ be a neighbourhood of $\xi$ in $X$ such that
$\overline{W_\xi}\subset V_\xi$. For any pair $x,y\in\Omega\setminus S$, we define the set
\begin{multline*}
  \quad\quad AG_{\Omega\setminus S}(x,y):=\{\gamma:[0,T]\lrarw\Omega\setminus S:\\
  \text{$\gamma$ is a $(\lambda,\kappa)$-almost-geodesic with respect to $K_{\Omega\setminus S}$ joining
  $x$ and $y$}\},
\end{multline*}
where $\lambda$ and $\kappa$ are as fixed above.}
\smallskip

We shall show that there exists a compact set $K\subset \Omega\setminus S$ such that for every $x\in
W_\xi\cap(\Omega\setminus S)$, every
$y\in V_\eta\cap(\Omega\setminus S)$, and every
$\gamma\in AG_{\Omega\setminus S}(x,y)$, we have $\langle\gamma\rangle\cap K\neq\emptyset$.
If possible, let this assertion be false. Let $(K_\nu)_{\nu\geq 1}$ be an exhaustion by compacts of
$\Omega\setminus S$. Then, there exist sequences 
$(x_\nu)_{\nu\geq 1}\subset W_\xi\cap(\Omega\setminus S)$,
$(y_\nu)_{\nu\geq 1}\subset V_\eta\cap(\Omega\setminus S)$, and $(\lambda,\kappa)$-almost-geodesics
$\gamma_\nu$ in $AG_{\Omega\setminus S}(x_\nu,y_\nu)$, for each $\nu\geq 1$, such that
$\langle\gamma_\nu\rangle\cap K_\nu=\emptyset$. At this point, we observe that by
\eqref{E:kobayshi-distance-set} and the fact that $k_{\Omega\setminus S} \geq k_{\Omega} $,
for each $\nu$
\[
  \gamma_{\nu}|_{I_{\nu}} \quad \text{is a $(\lambda,\kappa)$-almost-geodesic with respect to
  $K_\Omega$},
\]
where $I_{\nu}$ is any interval contained in ${\sf dom}(\gamma_{\nu})$. Furthermore
\[
  U_p\cap\bdy(\Omega\setminus S)=U_p\cap\bdy\Omega \quad\text{and}
  \quad V_{\eta}\cap (\Omega\setminus S) = V_{\eta}\cap \Omega.
\]
These two facts combined with an argument similar to the proof of the Claim made while
proving Theorem~\ref{T:visible-point-visible-domain} leads, essentially, to a similar
contradiction. This proves the present claim. \hfill $\btl$
\smallskip

Since $S\varsubsetneq X$ is closed, 
$\bdy(\Omega\setminus S)\setminus(\bdy\Omega\setminus S)\subseteq S$. Therefore, since $S$ is
totally disconnected, from the above claim it follows that $\bdy(\Omega\setminus
S)\setminus\vbdy(\Omega\setminus S)$
is totally disconnected. Therefore, by Theorem~\ref{T:visible-point-visible-domain} and
Corollary~\ref{C:visible-visibility-domain}, we can conclude that $\bdy(\Omega\setminus S)$ is visible.
\end{proof}

We conclude this section by presenting a proof of Proposition~\ref{P:simply-connected}.

\begin{proof}[The proof of Proposition~\ref{P:simply-connected}]
Since $\Omega$ is a simply connected domain, there exists a biholomorphic map $f:\D\lrarw\Omega$. Since
$\Omega$ is a Jordan domain, by \cite[Theorem~4.3.3]{braccicontrerasmasrigal:cshsmud20} we can conclude
that $f$ can be extended to a homeomorphism $\wt f:\overline\D\lrarw\Omc$, where $\Omc$ is the closure
of $\Omega$ viewed as a domain in $\widehat{\C}$. Fix
$p,q\in\bdy\Omega\subseteq\bdy_\infty\Omega$ such that $p\neq q$. We also fix $\lambda\geq 1$ and
$\kappa\geq 0$. Let $p'=\wt f^{-1}(p)\neq q'=\wt
f^{-1}(q)$. So, $p',q'\in\bdy\D$. Since $\bdy\D$ is visible, $(p',q')$ satisfies the
visibility
condition with respect to $K_\D$. Therefore, there exist neighbourhoods ${\wt U}_{p'}$ of $p'$ and
$\wt U_{q'}$ of $q'$ in $\C$ such that $\overline{\wt U_{p'}}\cap\overline{\wt {U}_{q'}}=\emptyset$ and
a compact set $\wt K$ of $\D$ such that the image of each $(\lambda,\kappa)$-almost-geodesic
$\wt\sigma:[0,T]\lrarw\D$ with $\wt\sigma(0)\in \wt U_{p'}$ and $\wt\sigma(T)\in \wt U_{q'}$ intersects
$\wt K$.
Consider two neighbourhoods $U_p$ of $p$ and $U_q$ of $q$
such that $\overline{ U_{p}}\cap\overline{ {U}_{q}}=\emptyset$ and such that
\[
  \wt f\big(\overline{\wt U_{p'}}\cap\overline\D\big)=\overline{U_p}\cap\overline\Omega \quad \text{and} 
  \quad \wt f\big(\overline{\wt {U}_{q'}}\cap\overline\D\big)=\overline{U_q}\cap\overline\Omega.
\]
Since $f$ is a biholomorphism, the 
pre-image (under $f$) of each $(\lambda,\kappa)$-almost-geodesic in $\Omega$ is a 
$(\lambda,\kappa)$-almost-geodesic in $\D$. Thus, by construction, the image of each 
$(\lambda,\kappa)$-almost-geodesic
$\sigma:[0,r]\lrarw\Omega$ with $\sigma(0)\in U_{p}$ and $\sigma(r)\in U_{q}$ intersects the compact
set, $f(\wt K)$. Hence the result.    
\end{proof}
\smallskip

\section{An application: A Wolff--Denjoy-type theorem}\label{S:application}
Recall that the classical Wolff--Denjoy Theorem is as follows:

\begin{result}[\cite{denjoy:sifa26, wolff:sgts26}]\label{R:WD}
Suppose $f:\D \lrarw \D$ is a holomorphic map then either:
\begin{enumerate}
 \item[$(a)$] $f$ has a fixed point in $\D$; or
 \item[$(b)$] there exists a point $p \in \bdy\D$ such that $\lim_{\nu \rightarrow \infty} f^{\nu}(x) = p$
 for any $x \in \D$, this convergence being uniform on compact subsets of $\D$.
\end{enumerate}
\end{result}

There is a rich history of generalisations of Result~\ref{R:WD}. However, until recently, these featured
various convex domains in place of $\D$, for instance: unit ball in $\Cn$ ($n\geq 2$)
\cite{herve:qpaadbde63}, bounded strongly convex domains in $\Cn$ \cite{abate:hihm88},
bounded convex domains in $\Cn$ with progressively weaker assumptions \cite{abateraissy:wdtncd14},
bounded strictly convex domains in a complex Banach space \cite{budzynskakucumowreich:DWtchmcBs13}.
To the best of my knowledge, the main
result in \cite{huang:ndpemihsm94} is the first result generalising Result~\ref{R:WD} to a class of 
contractible domains in
$\Cn$, $n\geq 2$, that includes domains that are not convex. Wolff--Denjoy-type theorems are now known
for certain metric spaces where $\bdy\D$ is replaced by some abstract 
boundary\,---\,see, for instance \cite{karlsson:nembf21, huczek:twdtsnmgs23}. The metrical
approaches to Wolff--Denjoy-type theorems motivated several generalisations of 
Result~\ref{R:WD} to visibility domains (in the sense of 
\cite{bharalimaitra:awnovfoewt21, bharalizimmer:gdwnv23}): see, for instance,
\cite{bharalizimmer:gdwnv17, bharalimaitra:awnovfoewt21, bharalizimmer:gdwnv23}. 
Our next theorem is in the spirit of the last three papers, extending the results therein to domains in
arbitrary complex manifolds. Wolff--Denjoy-type theorems for domains in manifolds, but with more restrictive 
hypotheses, have been established in \cite{chandelmitrasarkar:nvwrkdca21, vergamini:ttwpvc23}.

\begin{theorem}\label{T:wolff-denjoy}
Let $X$ be a complex manifold and $\Omega\varsubsetneq X$ be a taut domain such
that $\bdy\Omega$ is weakly visible. Let
$\Omc$ be some admissible compactification of $\overline\Omega$. Suppose $f:\Omega\lrarw\Omega$ is a
holomorphic map. Then, one of the following holds:
\begin{itemize}[leftmargin=25pt]
 \item[$(a)$] for every $x\in\Omega$, the set $\{f^{\nu}(x):\nu\geq 1\}$ is relatively compact in
 $\Omega$; or
 \item[$(b)$]there exists $p\in\bdy_\infty\Omega$ such that for every $x\in\Omega$, $f^\nu(x)$
 coverges locally uniformly to $p$ as $\nu\to\infty$.
\end{itemize}
\end{theorem}

In the statement of the above theorem, we have implicitly used the fact that when $\Omega$
is taut, then it is Kobayashi hyperbolic.
\smallskip

If $\Omega$ is not contractible, there is no reason to expect $f:\Omega\lrarw\Omega$ as in the above
theorem to have a fixed
point. In fact, in \cite{abateheinzer:hacdwfp92}, Abate--Heinzner construct \emph{contractible} domains 
$\Omega\varsubsetneq\Cn$, $n\geq 2$, and holomorphic maps $f:\Omega\lrarw\Omega$ that do not obey the dichotomy in
Result~\ref{R:WD}. Thus, in the generality we are interested in, $(a)$ of 
Theorem~\ref{T:wolff-denjoy} is the natural substitute of $(a)$ of Result~\ref{R:WD}.
\smallskip

In case Theorem~\ref{T:wolff-denjoy} seems rather abstract, owing to the involvement of an abstract
admissible
compactification of $\overline\Omega$, let us state a theorem with a more explicit hypothesis.
Wolff--Denjoy-type theorems are generally stated for bounded domains, but the following result shows
that the Wolff--Denjoy phenomenon extends to unbounded/non-relatively-compact domains (which is also
seen in the metrical versions of the Wolff--Denjoy phenomenon mentioned above).

\begin{theorem}\label{T:wolff-denjoy-unbdd}
Let $X$ be a non-compact complex manifold and $\Omega\varsubsetneq X$ be a taut
domain that is not relatively compact and such that $\bdy\Omega$ is weakly visible. Suppose
$f:\Omega\lrarw\Omega$ is a holomorphic map. Then, one of the following holds:
\begin{itemize}[leftmargin=25pt]
 \item[$(a)$] for every $x\in\Omega$, the set $\{f^{\nu}(x):\nu\geq 1\}$ is relatively compact in
 $\Omega$; or
 \item[$(b)$] either there exists $p\in\bdy\Omega$ such that for each $x\in\Omega$, $f^\nu(x)\to p$,
 or, for each $x\in\Omega$, the sequence $(f^{\nu}(x))_{\nu\geq 1}$ exits every compact subset 
 \textbf{of $\boldsymbol{X}$.}
 In the first case, $(f^{\nu})_{\nu\geq 1}$ converges locally uniformly to $p$ while, in the second case,
 $(f^{\nu})_{\nu\geq 1}$ is compactly divergent viewed as a sequence of maps into $X$.
\end{itemize}
\end{theorem}

Theorem~\ref{T:wolff-denjoy-unbdd} follows immediately from Theorem~\ref{T:wolff-denjoy} by taking
the admissible compactification of $\overline\Omega$ to be the closure of $\Omega$ in the
one-point compactification of $X$.
\smallskip

The argument for Theorem~\ref{T:wolff-denjoy} is similar to the argument in the proof of
\cite[Theorem~1.8]{bharalimaitra:awnovfoewt21}. However, as our setting is more general, we need
two lemmas that are analogous to \cite[Proposition~4.1]{bharalimaitra:awnovfoewt21} and \cite[Theorem~4.3]{bharalimaitra:awnovfoewt21}. 

\begin{lemma}\label{L:wolff-denjoy-l-1}
Let $X$ be a complex manifold and $\Omega\varsubsetneq X$ be a Kobayashi hyperbolic domain such that
$\bdy\Omega$ is weakly visible. Let
$\Omc$ be any admissible compactification of $\overline\Omega$. Suppose $(f_\nu)_{\nu\geq 1}$ is a
compactly divergent sequence in ${\rm Hol}(\Omega,\Omega)$. Then, there exists $p\in\bdy_\infty\Omega$
and a subsequence $(f_{\nu_j})_{j\geq 1}$ such that for all $x\in\Omega$ 
\[
  \lim_{j\to\infty}f_{\nu_j}(x)=p.
\]
\end{lemma}

\begin{proof}
By Corollary~\ref{C:visible-weak-visibility-domain}, we deduce that $\Omega$ is a weak
visibility domain subordinate to $\Omc$. At this stage, the argument proceeds exactly as in the proof of
\cite[Lemma~11.1]{bharalizimmer:gdwnv23} with $\bdy\overline\Omega^{End}$ replaced by  
$\bdy_\infty\Omega$, which works because charts are irrelevant to the latter proof. The proof of
\cite[Lemma~11.1]{bharalizimmer:gdwnv23} relies on the sequential compactness of the specific
admissible compactification $\overline\Omega^{End}$ and on examining maps defined on intervals of
$\R$, which carry over verbatim to our case.
\end{proof}

\begin{lemma}\label{L:wolff-denjoy-l-2}
Let $X$ be a complex manifold and $\Omega\varsubsetneq X$ be a Kobayashi hyperbolic domain such that
$\bdy\Omega$ is weakly visible. Let
$\Omc$ be any admissible compactification of $\overline\Omega$. Let $z_0\in\Omega$. Suppose
$f:\Omega\lrarw\Omega$ is a holomorphic map that satisfies
\[
  \limsup_{\nu\to\infty}K_\Omega(f^{\nu}(z_0),z_0)=\infty.
\]
Then, there exists $p\in\bdy_\infty\Omega$ such that for every sequence
$(\mu_k)_{k\geq 1}\subset\N$ that satisfies\linebreak
$\lim_{k\to\infty}K_\Omega(f^{\mu_k}(z_0),z_0)
=\infty$, we have
\[
  \lim_{k\to\infty}f^{\mu_k}(z_0)=p.
\]
\end{lemma}

\begin{proof}
From Corollary~\ref{C:visible-weak-visibility-domain} it follows that $\Omega$ is a weak visibility
domain
subordinate to $\Omc$. At this stage, we can essentially follow the same argument as in the proof of
\cite[Lemma~11.2]{bharalizimmer:gdwnv23} with $\bdy\overline\Omega^{End}$ replaced by  
$\bdy_\infty\Omega$\,---\,which works because the latter proof relies on the sequential compactness 
of the specific admissible compactification $\overline\Omega^{End}$ and on the behaviour of
$(1,1)$-almost-geodesics.
\end{proof}

\begin{result}[{\cite[Theorem~2.4.3]{abate:ithmtm89}}]\label{R:taut}
Let $X$ be a complex manifold such that $X$ is taut. Suppose $f:X\lrarw X$ is a holomorphic map. Then,
either of the following holds:
\begin{itemize}[leftmargin=25pt]
 \item[$(a)$]$(f^{\nu})_{\nu\geq 1}$ is relatively compact in
 ${\rm Hol}(X,X)$; or
 \item[$(b)$]$(f^\nu)_{\nu\geq 1}$ is compactly divergent.
\end{itemize}
\end{result}

Now, we are in a position to give a proof of Theorem~\ref{T:wolff-denjoy}.

\begin{proof}[The proof of Theorem~\ref{T:wolff-denjoy}]
Fix a Hermitian metric $h$ on $X$ and let $d_h$ be the distance induced by $h$ on $X$. 
Since $\Omega$ is taut, from Result~\ref{R:taut} it follows that either for every $x\in\Omega$, the set 
$\{f^{\nu}(x):\nu\geq 1\}$ is relatively compact in
$\Omega$ or $(f^\nu)_{\nu\geq 1}$ is compactly divergent. Therefore, we only need to show that $(b)$
follows when $(f^\nu)_{\nu\geq 1}$ is compactly divergent. Since $\bdy\Omega$ is weakly visible, by
Corollary~\ref{C:visible-weak-visibility-domain} we deduce that $\Omega$ is a weak
visibility domain subordinate to $\Omc$. Therefore, at this stage, the proof proceeds in a similar way as
the proof of \cite[Theorem~1.8]{bharalimaitra:awnovfoewt21}, with Lemma~\ref{L:wolff-denjoy-l-1}
serving in place of Montel's Theorem, Lemmas~\ref{L:wolff-denjoy-l-1}
and~\ref{L:wolff-denjoy-l-2} replacing any reference to \cite[Theorem~4.3]{bharalimaitra:awnovfoewt21}
and \cite[Proposition~4.1]{bharalimaitra:awnovfoewt21}, respectively, and the words ``$\Omega$ is a weak
visibility domain subordinate to $\Omc$" replacing the phrase ``$\Omega$ is a 
visibility domain". This works because the proof of \cite[Theorem~1.8]{bharalimaitra:awnovfoewt21} is 
metric-geometrical and charts are irrelevant to the latter proof. Essentially, the latter proof relies
on $(1,\kappa)$-almost-geodesics.
So, the only modifications that are required are the following:
\begin{itemize}
 \item We need to define a neighbourhood of $\xi$, which will be the analogue of
 $U_{\delta}$ appearing in the proof of \cite[Theorem~1.8]{bharalimaitra:awnovfoewt21},
 when $\xi\in\bdy_\infty\Omega$. Fix $\wt z\in\Omega$. For
 $\xi\in\Omc\setminus\emb(\overline\Omega)$, define
 \[
   U_\delta:= \text{the connected component of
   $\Omc\setminus \emb\big(\overline{B_{d_h}(\wt z,\delta^{-1})}\cap
   \overline{\Omega}\big)$ that contains $\xi$}.
 \]
 For $\xi\in\bdy\Omega$
 define $U_\delta:=B_{d_h}(\xi,\delta)\cap\overline\Omega$.
 \item Now, for arbitrary $\delta>0$ and $\xi\in\bdy_\infty\Omega$ define the function
 $G_\delta$ by 
 \[
   G_\delta(x_1,x_2):=\inf\{K_\Omega^m(x_1,x_2):\text{$m\in\N$, $f^m(x_1)\in U_\delta$}\}.
 \]
 \item For any arbitrary compact subset $K\subset \Omega$ and $\eta\in\bdy_\infty\Omega$, 
 we need to show that 
 \[
   \varepsilon:=\liminf_{z\to\eta}\inf_{k\in K} K_\Omega(k,z)
  \]
 is in $(0,\infty]$ (where $\eta$ has the same meaning as in \cite{bharalimaitra:awnovfoewt21}).
 Since $\Omega$
 is a Kobayashi hyperbolic domain, by \cite[Theorem~3.2.1]{kobayashi:hcs98} we have
 $Top(\Omega)=Top(K_\Omega)$ (recall: $Top(\Omega)$ denotes the manifold topology on $\Omega$). 
 The argument to show that $\varepsilon\in(0,\infty]$ in the proof of
 \cite[Theorem~1.14]{bharalizimmer:gdwnv23} relies only on the latter fact and, so, the
 same argument applies here.  
\end{itemize}
In view of the opening comments above, the result follows. 
\end{proof}
\smallskip

\section{Connections with Gromov hyperbolicity}\label{S:connnection-gromov-hyperbolic}
In this section, we shall follow the notation introduced in 
Sections~\ref{SS:Gromov-connection} and~\ref{SS:Gromov}.
However, we need some more notation and concepts before we can prove 
Theorem~\ref{T:gromov-boundary-visible}. These are taken from \cite{bharalizimmer:gdwnv23} but, 
in the interests of exposition that is reasonably self-contained, we present them here.

\begin{definition}\label{D:quasi-geodesic}
Let $X$ be a complex manifold and $\Omega\varsubsetneq X$ be a Kobayashi hyperbolic domain. Let
$I\subset\R$ be an interval. For $\kappa\geq 0$, a path $\sigma:I\lrarw\Omega$ is said to be a
\emph{$(1,\kappa)$-quasi-geodesic} if for all $s,t\in I$,
\[
  |s-t|-\kappa\leq K_\Omega (\sigma(s),\sigma(t))\leq |s-t|+\kappa.
\]
\end{definition}
 
\begin{definition}[see {\cite[Definition~10.2]{bharalizimmer:gdwnv23}}]\label{D:admissible-family}
Let $X$ be a complex manifold and $\Omega\varsubsetneq X$ be a Kobayashi hyperbolic domain. Fix 
$\kappa\geq 0$. We shall say that a family $\mathcal{F}$ of continuous $(1,\kappa)$-quasi-geodesics is a 
\emph{$\kappa$-admissible family} if
\begin{itemize}[leftmargin=25pt]
  \item[$(a)$]for each pair of distinct points $x,y\in\Omega$, there exists a continuous 
  $(1,\kappa)$-quasi-geodesic $\sigma\in\mathcal{F}$ joining $x$ and $y$,
  \item[$(b)$]if a continuous $(1,\kappa)$-quasi-geodesic $\sigma:[0,T]\lrarw\Omega$ is in $\mathcal{F}$, then so
  is $\overline{\sigma}$, where
  \[
    \overline{\sigma}(t):=\sigma(T-t),\quad t\in[0,T].
  \]
\end{itemize}
\end{definition}

For $x,y\in\Omega$, we shall use the notation $\sigma|x\frown y$ to denote that $\sigma$ is a continuous
$(1,\kappa)$-quasi-geodesic (for some $\kappa\geq 0$) joining $x$ and $y$. In view of
Proposition~\ref{P:almost-geodesic}, for any $\kappa>0$, there exists a non-empty $\kappa$-admissible
family. Therefore, the above discussion is naturally relevant to the ideas we will be discussing in this
section.
\smallskip

We can now give the following definition, wherein we will commit the following abuse of notation: for
any curve $\sigma$, we shall write the \textbf{image} of $\sigma$ simply as $\sigma$.

\begin{definition}\label{D:rips-condition}
Let $X$ be a complex manifold and $\Omega\varsubsetneq X$ be a Kobayashi hyperbolic domain. 
Fix $\kappa\geq 0$.
\begin{itemize}[leftmargin=25pt]
  \item[$(a)$]A \emph{$\kappa$-triangle} is a triple of continuous $(1,\kappa)$-quasi-geodesics 
  $\alpha|x\frown o$, $\beta|y\frown o$, $\gamma|x\frown y$, where $x,y,o\in\Omega$. We call $x,y,o$
  the \emph{vertices}, and $\alpha,\beta,\gamma$ the \emph{sides} of the $\kappa$-triangle.
  \item[$(b)$] Let $\delta\geq 0$. A $\kappa$-triangle $\Delta$ is said to be \emph{$\delta$-slim} if,
  for each side $\sigma$ of $\Delta$,
  \[
    \sigma\subset\bigcup_{\Sigma\in\Delta\setminus\{\sigma\}}
    \{z\in\Omega:K_\Omega(z,\Sigma)\leq\delta\}.
  \]
  \pagebreak
  \item[$(c)$]Let $\mathcal{F}$ be a $\kappa$-admisible family of continuous $(1,\kappa)$-quasi-geodesics. Let
  $\delta\geq 0$. We say that $\Omega$ satisfies the \emph{$(\mathcal{F},\kappa,\delta)$-Rips condition}
  if every $\kappa$-triangle in $\Omega$ whose sides are (the images of) continuous $(1,\kappa)$-quasi-geodesics
  belonging to $\mathcal{F}$ is $\delta$-slim.
\end{itemize}
\end{definition}

The result that we need, and which the above notation serves, is the following:

\begin{result}[paraphrasing {\cite[Theorem~10.6]{bharalizimmer:gdwnv23}}]\label{R:Gromov-hyp-rips-condition}
Let $X$ be a connected complex manifold and $\Omega\varsubsetneq X$ be a Kobayashi hyperbolic domain.
Let $\kappa\geq 0$ be such that there exists some $\kappa$-admissible family of 
continuous $(1,\kappa)$-quasi-geodesics. Suppose $(\Omega,K_\Omega)$ is $\delta$-hyperbolic for some 
$\delta\geq 0$. Then, for any $\kappa$-admissible family $\mathcal{F}$, $\Omega$ satisfies the
$(\mathcal{F},\kappa,3\delta+6\kappa)$-Rips condition.
\end{result}

The above result paraphrases \cite[Theorem~10.6]{bharalizimmer:gdwnv23} in two ways. Firstly,
\cite[Theorem~10.6]{bharalizimmer:gdwnv23} has been stated for domains $\Omega\varsubsetneq\Cn$. Secondly,
a partial converse of the above result, with $X=\Cn$, forms a part of 
\cite[Theorem~10.6]{bharalizimmer:gdwnv23}. We have not stated (the relevant version of) the latter because
it is not needed here. The proof of \cite[Theorem~10.6]{bharalizimmer:gdwnv23} extends verbatim to the
setting of Result~\ref{R:Gromov-hyp-rips-condition} because this proof relies on viewing $\Omega$ merely
as a metric space, i.e., $(\Omega,K_\Omega)$, and on the fact that the  quasi-geodesics
considered are continuous (in \cite{bharalizimmer:gdwnv23}, continuous $(1,\kappa)$-quasi-geodesics are
called ``$(1,\kappa)$-near-geodesics'').
\smallskip

In the proof that follows:
\begin{itemize}
  \item The symbol $\emb$ will denote the embedding
  $\Omega\hookrightarrow\overline\Omega^G$ that is a part of the definition of $\overline\Omega^G$.
  \item $\lim$ will denote limits in all contexts other than limits with respect to the
  topology on
  $\overline\Omega^G$.
  \item Given $x, y, o\in \Omega$, the Gromov product with the underlying distance being
  $K_{\Omega}$ will be denoted by $(x|y)_o^{\Omega}$.
\end{itemize}
\smallskip

Also, we shall commit the (mild) abuse of notation described in Section~\ref{SS:visi-compacti-WD}. With
these words, we can now present

\begin{proof}[The proof of Theorem~\ref{T:gromov-boundary-visible}]
Fix $\xi,\eta\in\bdy_G\Omega$ such that $\xi\neq\eta$. We also fix $\kappa\geq 0$. Consider
sequences $(x_\nu)_{\nu\geq 1}\subset\Omega$ and $(y_\nu)_{\nu\geq 1}\subset\Omega$ such that
$x_\nu\xrightarrow{G}\xi$, and $y_\nu\xrightarrow{G}\eta$.
For any pair $x,y\in\Omega$
we define the set (with $\kappa$ understood to be fixed as above)
\[
  AG_{\Omega}(x,y):=\{\gamma:[0,T]\lrarw\Omega:
  \text{$\gamma$ is a $(1,\kappa)$-almost-geodesic joining $x$ and $y$}\}.
\]
\smallskip

We shall show that there exists a compact set $K\subset \Omega$ such that for every $\nu\geq 1$, and
every $\gamma\in AG_\Omega(x_\nu,y_\nu)$, we have $\langle\gamma\rangle\cap K\neq\emptyset$. If
possible, let this assertion be false. Let $(K_\nu)_{\nu\geq 1}$ be an exhaustion by compacts of
$\Omega$. Then, passing to a subsequence and relabelling, if needed, we may assume, for each 
$\nu\geq 1$, there exists a $(1,\kappa)$-almost-geodesic $\gamma_\nu\in AG_\Omega(x_\nu,y_\nu)$ such 
that $\langle\gamma_\nu\rangle\cap K_\nu=\emptyset$. Since $\overline\Omega^G$ is locally compact,
there exists a $\overline\Omega^G$-open neighbourhood $V_\xi$ of $\xi$ such that $\overline{V_\xi}^G$ is
compact. Fix neighbourhoods $W_\xi$, $W_{\xi}^1$, $W_{\xi}^2$ of $\xi$ such that 
$W_\xi\Subset W_\xi^1\Subset
W_\xi^2\Subset V_\xi.$ Without loss of generality, we may assume that $x_\nu\in W_\xi$ for all $\nu\geq 1$. 
Let $r_\nu>0$ and ${y'_\nu}=\gamma_\nu(r_\nu)$ be such that ${y'_\nu}\in
\emb^{-1}\big(W_\xi^2\setminus\big(\overline{W_\xi^1}^G\bigcup\bdy_G\Omega\big)\big)$, and
such that the image of $\wt{\gamma}_{\nu}:=\gamma_\nu{\mid}_{[0,r_\nu]}$ is contained in $W_\xi^2$.
Since $\overline{V_\xi}^G$ is compact, by passing to a subsequence and relabelling, if needed, we may
assume that there exists a point $\eta'\in\overline\Omega^G$ such that $y'_\nu\xrightarrow{G}\eta'$.
Since $y'_\nu\notin K_\nu$ for each $\nu\geq 1$, we have $\eta'\in\bdy_G\Omega$. By
construction, $\xi\neq\eta'$. Therefore, by Part $(b)$ of Result~\ref{R:gromov-topology-property}, it
follows that $(x_\nu)_{\nu\geq 1}$ and $(y_\nu)_{\nu\geq 1}$ are Gromov sequences and 
$\xi=[x_\nu]\neq[y'_\nu]=\eta'$. Therefore, by Definition~\ref{D:gromov-product}, there exists a
constant $M>0$ such that $(x_\nu|y'_\nu)_o^\Omega<M$ for all $\nu\geq 1$.
Let $\kappa'>\kappa$. Define $\mathscr{F}:=\{\wt\gamma_\nu:\nu\geq 1\}$. Since each $\wt\gamma_\nu$ is 
also a $(1,\kappa')$-almost-geodesic in $\Omega$, $\mathscr{F}$ is a subset of a $\kappa'$-admissible
family, namely the collection of all $(1,\kappa')$-almost-geodesics in $\Omega$. Denote the latter family
as $\mathcal{F}$.
\smallskip

Fix
a point $o\in\Omega$. Since $\kappa'>0$, in view of Proposition~\ref{P:almost-geodesic}, for each
$\nu\geq 1$ we can fix $\alpha_\nu$, a $(1,\kappa')$-almost-geodesic joining $x_\nu$ and $o$, and fix
$\beta_\nu$, 
a $(1,\kappa')$-almost-geodesic joining $y'_\nu$ and $o$. 
Let $\Delta_\nu$ denote the $\kappa'$-traingle whose sides are $\alpha_\nu$, $\beta_\nu$, and 
$\wt\gamma_\nu$. Let $\delta\geq 0$ be such that $(\Omega,K_\Omega)$ is $\delta$-hyperbolic. Let
$\delta':=3\delta+6\kappa'$. By Result~\ref{R:Gromov-hyp-rips-condition}, $\Omega$ satisfies the
$(\mathcal{F},\kappa',\delta')$-Rips condition. Therefore,  for each $\nu\geq 1$,
\[
  \wt\gamma_\nu\subset\bigcup_{\Sigma\in\Delta_\nu\setminus\{\wt\gamma_\nu\}}
  \{z\in\Omega:K_\Omega(z,\Sigma)\leq\delta'\}.
\]
Since $\wt\gamma_\nu$ is connected, for each $\nu$, there exists 
$z_\nu\in\wt\gamma_\nu$ such that $z_\nu\in\{z\in\Omega:K_\Omega(z,\alpha_\nu)\leq\delta'\}\cap
\{z\in\Omega:K_\Omega(z,\beta_\nu)\leq\delta'\}$. For each $\nu\geq 1$, fix $a_\nu\in\alpha_\nu$ and
$c_\nu\in\beta_\nu$
such that $K_\Omega(a_\nu,z_\nu)$, $K_\Omega(c_\nu,z_\nu)\leq\delta'$. 
By triangle inequality it follows that 
\begin{align}\label{E:gromov-hyperbolic-theorem}
 K_\Omega(o,a_\nu)\geq K_\Omega(o,z_\nu)-\delta',\\
 K_\Omega(o,c_\nu)\geq K_\Omega(o,z_\nu)-\delta'\label{E:gromov-hyperbolic-theorem-2}.
\end{align}
Therefore, from 
\eqref{E:gromov-hyperbolic-theorem}, and since $\alpha_\nu$ is a $(1,\kappa')$-almost-geodesic, we have 
\begin{align*}
 K_\Omega(o,a_\nu)+K_\Omega(a_\nu,x_\nu)&\leq K_\Omega(o,x_\nu)+3\kappa'\\ 
 \implies K_\Omega(a_\nu,x_\nu)&\leq K_\Omega(o,x_\nu)-K_\Omega(o,z_\nu)+\delta'+3\kappa'.
\end{align*}
Therefore, for each $\nu\geq 1$,
\begin{align*}
 K_\Omega(x_\nu,z_\nu)&\leq K_\Omega(a_\nu,x_\nu)+K_\Omega(a_\nu,z_\nu)\\&\leq K_\Omega(o,x_\nu)-
 K_\Omega(o,z_\nu)+2\delta'+3\kappa'.  
\end{align*}
Similarly, from 
\eqref{E:gromov-hyperbolic-theorem-2}, and since $\beta_\nu$ is a $(1,\kappa')$-almost-geodesic, we have
$K_\Omega(y'_\nu,z_\nu)\leq K_\Omega(o,y'_\nu)-
K_\Omega(o,z_\nu)+2\delta'+3\kappa'$, for all $\nu\geq 1$. Combining last two inequalities, it
follows that for all
$\nu\geq 1$,
\begin{align}
 2K_\Omega(o,z_\nu)&\leq 2(x_\nu|y'_\nu)_o^\Omega+4\delta'+6\kappa'< 2M+4\delta'+6\kappa'.
 \label{E:inequality-finite-gromov} 
\end{align} 
Now, since $(z_\nu)_{\nu\geq 1}\subset V_\xi$ and $\overline{V_\xi}^G$ is compact, there exists a
subsequence
$(\nu_j)_{j\geq 1}$ and $\xi'\in\overline{V_\xi}^G$ such that $z_{\nu_j}\xrightarrow{G}\xi'$. Now,
since $z_{\nu_j}\notin K_{\nu_j}$ for all $j\geq 1$, we have $\xi'\in\bdy_G\Omega$. Therefore, by
Part~$(b)$ of Result~\ref{R:gromov-topology-property},
$(z_{\nu_j})_{j\geq 1}$ is a Gromov sequence. Hence, from Definition~\ref{D:gromov-product}, we have 
$\lim_{j\to\infty}K_\Omega(o,z_{\nu_j})=\infty$, which 
contradicts \eqref{E:inequality-finite-gromov}. Hence the result.

\end{proof}

Now, we will give a proof of Proposition~\ref{P:gromov-visible-example}. Let $D$ be as in the
statement of this proposition. The map
$\emb:D\hookrightarrow\overline D^G$ is as explained above.

\begin{proof}[The proof of Proposition~\ref{P:gromov-visible-example}]
Since $D$ is Kobayashi hyperbolic, by the distance decreasing property for the Kobayashi
distance of the inclusion map $\Omega\hookrightarrow D$, it follows that $\Omega$ is a Kobayashi
hyperbolic 
domain. Since $K$ is compact and $\mathcal{H}^{2n-2}(K)=0$,
by \cite[Theorem~3.2.22]{kobayashi:hcs98} we have
\begin{align}\label{E:kobayshi-distance-set-gromov}
 K_{\Omega}=K_D{\mid}_{\Omega\times\Omega}.
\end{align}
Since $(D,K_D)$ is Gromov hyperbolic, there exists $\delta\geq 0$ such that $(D,K_D)$ is
$\delta$-hyperbolic. Therefore, by Definition~\ref{D:Gromov-hyperbolic} and 
\eqref{E:kobayshi-distance-set-gromov}, $(\Omega,K_\Omega)$ is
$\delta$-hyperbolic. This proves part $(i)$. 
\smallskip

Now, to prove part~$(ii)$, since $\overline D$ is homeomorphic to $\overline D^G$, it is enough
to prove that $\overline\Omega^G$ is homeomorphic to
$\overline D^G\setminus\emb(K)$. Since $K$ is compact, as a \textbf{set,}
\begin{equation}\label{E:as_a_set}
  \overline D^G\setminus\emb(K) = \Omega\cup\bdy_G D.
\end{equation}
From \eqref{E:kobayshi-distance-set-gromov} and Definition~\ref{D:gromov-product}, it
follows that every equivalence class of Gromov sequences of $\Omega$ is an equivalence class of Gromov 
sequences of $D$ and the vice versa (since $K$ is compact). Therefore, as a set 
$\overline\Omega^G=\overline D^G\setminus\emb(K)$. Now, since the Gromov topology depends only on the
Kobayashi distance, by Result~\ref{R:gromov-topology-property}, \eqref{E:as_a_set}, and 
\eqref{E:kobayshi-distance-set-gromov} we have that $id_{D\setminus K}$ extends to
a homeomorphism between $\overline\Omega^G$ and $\overline D^G\setminus\emb(K)$. Therefore, clearly,
$\overline\Omega^G$ is non-compact and locally compact.
\end{proof}
\smallskip

\section*{Acknowledgements}
I would like to thank my thesis advisor, Prof. Gautam Bharali, 
for many invaluable discussions during course of this work.
I am also grateful to him for his advice on the writing of this paper. I am especially grateful to
the anonymous referees of this work for pointing out the need for some corrections to the initial draft
and for their helpful suggestions on exposition. This work is supported in part
by a scholarship from the Prime
Minister's Research Fellowship (PMRF) programme (fellowship no.~0201077) 
and by a DST-FIST grant (grant no.~DST FIST-2021 [TPN-700661]).

\end{document}